\documentclass[10pt]{amsart}
\usepackage[nohug,heads=littlevee]{diagrams}
\usepackage{mathrsfs}
\usepackage{bm}
\usepackage{amssymb}
\usepackage[latin1]{inputenc}
\usepackage{xr-hyper}
\usepackage[pagebackref=true,plainpages=false,pdfpagelabels]{hyperref}
\usepackage[nobysame,alphabetic]{amsrefs}
\usepackage[left=3.5cm,top=2.5cm,right=3.6cm,bottom=1.4in,asymmetric]{geometry}
\usepackage{mathtools}
\usepackage{enumitem}
\newcommand{\defnword}[1]{\textbf{#1}}
\newcommand{\comment}[1]{}
\newcommand{\on}[1]{\operatorname{#1}}
\newcommand{\bb}[1]{\mathbb{#1}}
\newcommand{\mb}[1]{\mathbf{#1}}
\newcommand{\mf}[1]{\mathfrak{#1}}
\newcommand{\ord}{\on{ord}}

\numberwithin{equation}{subsection}
\newtheorem{introthm}{Theorem}
\newtheorem{introcor}[introthm]{Corollary}

\newtheorem{prp}[subsection]{Proposition}
\newtheorem{thm}[subsection]{Theorem}
\newtheorem*{thm*}{Theorem}
\newtheorem{lem}[subsection]{Lemma}
\newtheorem*{lem*}{Lemma}
\newtheorem{corollary}[subsection]{Corollary}
\theoremstyle{definition}
\newtheorem{defn}[subsection]{Definition}

\theoremstyle{remark}

\newtheorem{rem}[subsection]{Remark}
\newtheorem{assump}[subsection]{Assumption}
\newtheorem*{assump*}{Assumption}

\newarrow{Equals}{=}{=}{}{=}{}
\newcommand{\Square}[8]{\begin{diagram}
  #1&\rTo^{#2}&#3\\
  \dTo^{#4}&&\dTo_{#5}\\
  #6&\rTo_{#7}&#8
\end{diagram}.
}

\newcommand{\Real}{\mathbb{R}}
\newcommand{\Int}{\mathbb{Z}}

\newcommand{\Comp}{\mathbb{C}}

\newcommand{\Adele}{\mathbb{A}}
\newcommand{\Field}{\mathbb{F}}
\newcommand{\Rat}{\mathbb{Q}}

\newcommand{\rk}{\on{rk}}

\newcommand{\im}{\on{im}}
\newcommand{\gr}{\on{gr}}
\newcommand{\into}{\hookrightarrow}

\newcommand{\Def}{\on{Def}}

\newcommand{\pr}{\bm{\pi}}

\newcommand{\Rg}{\mathscr{O}}
\newcommand{\Reg}[1]{\Rg_{#1}}

\newcommand{\Hom}{\on{Hom}}
\newcommand{\End}{\on{End}}

\newcommand{\Aut}{\on{Aut}}

\newcommand{\Gal}{\on{Gal}}

\newcommand{\Pic}{\on{Pic}}

\newcommand{\cris}{\on{cris}}
\newcommand{\dR}{\on{dR}}

\newcommand{\Bcris}{B_{\cris}}
\newcommand{\Bdr}{B_{\dR}}

\newcommand{\pow}[1]{[\vert#1\vert]}

\newcommand{\Spec}{\on{Spec}}
\newcommand{\Spf}{\on{Spf}}

\newcommand{\et}{\on{\acute{e}t}}

\newcommand{\KS}{\on{KS}}

\newcommand{\Fil}{\on{Fil}}
\newcommand{\Fr}{\on{Fr}}

\newcommand{\mx}{\mathfrak{m}}

\newcommand{\dual}[1]{{#1}^{\vee}}

\newcommand{\Sh}{\on{Sh}}
\newcommand{\Ss}{\mathscr{S}}

\newcommand{\an}{\on{an}}
\newcommand{\nr}{\on{nr}}

\newcommand{\GL}{\on{GL}}

\newcommand{\SO}{\on{SO}}

\newcommand{\GSpin}{\on{GSpin}}

\begin{document}
\title[The Tate conjecture for K3 surfaces]{The Tate conjecture for K3 surfaces in odd characteristic}
\author{Keerthi Madapusi Pera}
\address{Keerthi Madapusi Pera\\%
Department of Mathematics\\%
1 Oxford St\\%
Harvard University\\%
Cambridge, MA 02118\\%
USA}
\email{keerthi@math.harvard.edu}

\begin{abstract}
  We show that the classical Kuga-Satake construction gives rise, away from characteristic $2$, to an open immersion from the moduli of primitively polarized K3 surfaces (of any fixed degree) to a certain regular integral model for a Shimura variety of orthogonal type. This allows us to attach to every polarized K3 surface in odd characteristic an abelian variety such that divisors on the surface can be identified with certain endomorphisms of the attached abelian variety. In turn, this reduces the Tate conjecture for K3 surfaces over finitely generated fields of odd characteristic to a version of the Tate conjecture for certain endomorphisms on the attached Kuga-Satake abelian variety, which we prove. As a by-product of our methods, we also show that the moduli stack of primitively polarized K3 surfaces of degree $2d$ is quasi-projective and, when $d$ is not divisible by $p^2$, is geometrically irreducible in characteristic $p$. We indicate how the same method applies to prove the Tate conjecture for co-dimension $2$ cycles on cubic fourfolds.
\end{abstract}

\maketitle

\section*{Introduction}\label{sec:intro}

The goal of this paper is to prove:
\begin{introthm}\label{intro:thm:main}
Let $X$ be a K3 surface over a finitely generated field $k$ of characteristic not equal to $2$. Then the Tate conjecture holds for $X$.

That is, for any prime $\ell$ invertible in $k$, the $\ell$-adic Chern class map
\[
 \Pic(X)\otimes\Rat_{\ell}\xrightarrow{\on{ch}}H^2_{\et}\bigl(X_{k^{\on{sep}}},\Rat_{\ell}(1)\bigr)^{\Gamma}
\]
is an isomorphism. Here, $k^{\on{sep}}$ is a separable closure of $k$ and $\Gamma=\Gal(k^{\on{sep}}/k)$ is the associated absolute Galois group.
\end{introthm}

\comment{When $k$ is finite, the following statements (and many more) are equivalent to the Tate conjecture for $X$~\cites{tate:motives1,milne:artin-tate}:
\begin{enumerate}
  \item The crystalline Chern class map
  \[
   \Pic(X)\otimes\Rat_p\xrightarrow{\on{ch}}H^2_{\cris}\bigl(X/W(k)\bigr)^{F=p}_{\Rat}
  \]
  is an isomorphism.
  \item For some prime $\ell$, the $\ell$-torsion $\on{Br}(X)[\ell]$ of the Brauer group of $X$ is finite.
  \item The Brauer group $\on{Br}(X)$ is finite.
  \item The order of the pole at $s=1$ of the zeta function $Z(X,q^{-s})$ is equal to the rank of $\Pic(X)$.
\end{enumerate}
}
Work of Lieblich-Maulik-Snowden~\cite{lms} shows that Theorem \ref{intro:thm:main} implies:
\begin{introcor}\label{intro:cor:lms}
There are only finitely many isomorphism classes of K3 surfaces over a finite field of odd characteristic.
\end{introcor}\comment{In \emph{loc. cit}, one finds the restriction that the characteristic $p$ is at least $5$. This is not necessary; cf.~(\ref{ks:cor:lms}).}

The following cases of Theorem \ref{intro:thm:main} are already known:
\begin{enumerate}
  \item When the field $k$ is of characteristic $0$: cf.~\cite{tate:motives1}*{Theorem 5.6(a)} or \cite{andre:shaf}.
  \item When $k$ is finite of characteristic at least $5$: This is due to Nygaard and Nygaard-Ogus~\cites{nygaard,nygaard_ogus}\footnote{The result of \cite{nygaard} for ordinary K3 surfaces does not appear to have any restriction on the characteristic.} for K3 surfaces of finite height, and Maulik~\cite{maulik} and Charles~\cite{charles} for supersingular K3 surfaces. Maulik's work utilizes the case of elliptic K3 surfaces, due to Artin-Swinnerton-Dyer~\cite{artin-sdyer}, but Charles's is independent of it, being an application of a general result for reductions of holomorphic symplectic varieties.
\end{enumerate}

The main contribution of this article is an unconditional proof of the conjecture in odd characteristic. Our methods are independent of the results above, but owe a substantial spiritual debt to the proof in characteristic $0$, which combines the classical Kuga-Satake construction with Deligne's theory of absolute Hodge cycles and Faltings's isogeny theorem.
\subsection*{Kuga-Satake construction}
In characteristic $0$, the Kuga-Satake construction attaches to every polarized K3 surface $(X,\xi)$ a polarized abelian variety $A$ such that the primitive cohomology group $PH^2(X,\xi)$ embeds within $H^1(A)\otimes H^1(A)$ as a sub-Hodge structure. One can extend this construction to finite characteristic as in \cite{deligne:k3weil}, by lifting to characteristic $0$, applying the Kuga-Satake construction, and taking its reduction. The crystalline compatibility (up to isogeny) of such a construction is shown in \cite{ogus:duke}*{\S~7}. We make two improvements to this: First, we show the crystalline compatibility on an \emph{integral} level. Second, we show that the Kuga-Satake construction sees enough geometry to allow us to view divisors on the K3 surface $X$ as endomorphisms of $A$. This is of course predicted by the conjecturally motivic nature of the construction.

In particular, we can reduce the Tate conjecture for $X$ to a refined version of Tate's theorem for endomorphisms of $A$.

The reader is directed to (\ref{ks:thm:ksconst}) in the body of the paper for a precise version of the following result:
\begin{introthm}\label{intro:thm:ksconst}
Given any field $k$ of odd characteristic $p$ and a polarized K3 surface $(X,\xi)$ over $k$, there exists a finite separable extension $k'/k$ and an abelian variety $A$ over $k'$, the \defnword{Kuga-Satake} abelian variety such that the $\Int_{\ell}$ and crystalline realizations of the primitive cohomology $PH^2(X,\xi)$ embed naturally within those of $H^1(A)\otimes H^1(A)$. Moreover, there is a canonical inclusion
\[
\Pic(X_{k'})\supset\langle\xi\rangle^{\perp}\into\End(A)
\]
compatible, via the cycle class maps, with the corresponding embeddings of cohomology groups. Its image consists of those endomorphisms whose cohomological realizations in $H^1(A)\otimes H^1(A)(1)$ lie in the image of $PH^2(X,\xi)(1)$.
\end{introthm}

It is essential for our method that we work with \emph{families} of K3 surfaces: We view the Kuga-Satake correspondence in characteristic $0$ as a period map from the moduli of polarized K3 surfaces to an appropriate orthogonal Shimura variety, and use the theory of integral models from~\cite{mp:reg} to extend it to period map over $\Int[2^{-1}]$. The integral crystalline compatibility of this construction shows that the period map is \'etale. This permits us to prove the inclusion $\langle\xi\rangle^{\perp}\into\End(A)$ of the theorem by lifting---one divisor at a time---to characteristic $0$, where we can appeal to the Lefschetz (1,1) theorem.

\subsection*{The Tate conjecture for special endomorphisms}
Given the above theorem, it is natural to make the following definition: A \defnword{special endomorphism} of the Kuga-Satake abelian variety $A$ is an element $f\in\End(A)$ whose cohomological realizations in $H^1(A)\otimes H^1(A)(1)$ lie in the image of $PH^2(X,\xi)(1)$. We will write $L(A)$ for the space of special endomorphisms. When $k$ is finitely generated, the Tate conjecture for $(X,\xi)$ now reduces to the statement that $L(A)$ has the expected rank.

This last assertion is best viewed in the setting of motives attached to points of orthogonal Shimura varieties. Such varieties are attached to quadratic lattices over $\Int$ of signature $(n,2)$. For instance, the one that appears as the target of the period map mentioned above is attached to the primitive cohomology lattice of a polarized K3 surface; it has signature $(19,2)$. Given a lattice $L$ of signature $(n,2)$, the associated Shimura variety $\Sh(L)$ is $n$-dimensional and defined over $\Rat$. The theory of \cite{mp:reg}, which builds on work of Kisin~\cite{kis3}, provides us with a regular integral model $\Ss(L)$ over $\Int\left[\frac{1}{2}\right]$.

To every geometric point $s\to\Ss(L)$, we can attach an abelian variety $A^{\KS}_s$, again called the Kuga-Satake abelian variety. Suppose that $k(s)$ has characteristic $p>2$. Then this abelian variety comes equipped with a right action of the Clifford algebra $C(L)$, as well as the following additional structure:
\begin{itemize}
  \item For every $\ell\neq p$, a distinguished sub-space
 \[
      \bm{V}_{\ell,s}\subset\End_{C(L)}\bigl(H^1_{\et}(A^{\KS}_s,\Rat_{\ell})\bigr)
 \]
  such that, for all $f\in\bm{V}_{\ell,s}$, the composition $f\circ f$ is a scalar. The space $\bm{V}_{\ell,s}$ with the quadratic form $f\mapsto f\circ f$ is isometric to $L\otimes\Rat_{\ell}$.
  \item A distinguished sub-$F$-isocrystal
   \[
      \bm{V}_{\cris,s}\subset\End_{C(L)}\bigl(H^1_{\cris}(A^{\KS}_s/W(k(s)))\bigr)_{\Rat}
   \]
  such that, for all $f\in\bm{V}_{\cris,s}$, the composition $f\circ f$ is a scalar. The space $\bm{V}_{\cris,s}$ with the quadratic form $f\mapsto f\circ f$ is isometric to $L\otimes W(k(s))$.
\end{itemize}
We now define the space of special endomorphisms $L(A^{\KS}_s)$ to be the sub-space of $\End\bigl(A^{\KS}_s\bigr)$ consisting of those elements whose cohomological realizations land in the distinguished sub-spaces given above.

When $L$ is the primitive cohomology lattice of a polarized K3 surface and $s$ arises from a polarized K3 $(X,\xi)$, $A^{\KS}_s$ is just the associated Kuga-Satake abelian variety $A$, and the distinguished sub-spaces $\bm{V}_{\ell,s}$ and $\bm{V}_{\cris,s}$ can be identified with the realizations of the primitive cohomology $PH^2(X,\xi)$. So the general definition recovers our definition from this special case.

Suppose now that $s$ is defined over a finitely generated extension $k$. Then the distinguished sub-spaces $\bm{V}_{\ell,s}$ are stable under the action of $\Gamma$, the absolute Galois group of $k$. We will assume that $A^{\KS}_s$ and all of its endomorphisms are also defined over $k$. The key technical result of this paper is:
\begin{introthm}\label{intro:thm:tatespecial}
Under a certain $\ell$-independence condition, for every $\ell\neq p$, the natural map of $\ell$-adic vector spaces
\[
 L(A^{\KS}_s)\otimes\Rat_{\ell}\to \bm{V}_{\ell,s}^{\Gamma}
\]
is an isomorphism.
\end{introthm}
The $\ell$-independence condition essentially says that the dimension of the invariant sub-spaces $\bm{V}_{\ell,s}^{\Gamma}$ does not depend on $\ell$; cf. Section~\ref{sec:special}. In the situation of the Kuga-Satake abelian variety attached to a polarized K3 surface, this condition always holds, and so we obtain the Tate conjecture for K3s as a consequence.

The most important case of the theorem is when $s$ is defined over a finite field. We can deduce the general result from this by invoking Zarhin's theorem for endomorphisms of abelian varieties over finitely generated fields and a specialization argument.

For the proof in the finite field case, we begin with a simple observation. Let $I$ be the largest algebraic sub-group of $\underline{\Aut}^{\circ}_{C(L)}(A^{\KS}_s)$ (viewed as the algebraic group attached to the group of units of the algebra $\End_{C(L)}\bigl(A^{\KS}_s\bigr)_{\Rat}$) that stabilizes the distinguished sub-spaces $\bm{V}_{\ell,s}$ and $\bm{V}_{\cris,s}$. Then, for every $\ell\neq p$, the map considered in Theorem~\ref{intro:thm:tatespecial} is a map of representations of the $\Rat_{\ell}$-group $I_{\Rat_{\ell}}$. We will be done if we can prove two assertions: First, for \emph{some} $\ell\neq p$, $\bm{V}_{\ell,s}^{\Gamma}$ is irreducible as a representation of $I_{\Rat_{\ell}}$. Second, $L(A^{\KS}_s)\neq 0$.

\comment{For $\ell\neq p$, we define $I_{\ell}$ to be the stabilizer in $G_{\ell}=\GSpin(L\otimes\Rat_{\ell})$ of the Frobenius endomorphism of $A^{\KS}_s$ (with respect to $k$): It is a Levi sub-group of $G_{\ell}$ and acts irreducibly on $V_{\ell,s}^{\Gamma}$. Then $I_{\Rat_{\ell}}$ is naturally a sub-group of $I_{\ell}$. So, to finish the proof of the first assertion, we only need to find an $\ell$ such that $I_{\Rat_{\ell}}=I_{\ell}$. Since $I_{\Rat_{\ell}}$ is reductive, it suffices to show that it contains a Borel sub-group of $I_{\ell}$.}
When $\ell$ is a prime such that $G_{\ell}$ is split, we show the first assertion using a result of Kisin~\cite{kis:langrap}.\comment{It is a group-theoretic reinterpretation of the original argument of Tate in the proof of his conjecture for endomorphisms of abelian varieties over finite fields~\cite{tate:end}, and uses in an essential way Kisin's construction of integral canonical models.}

Our method for showing that $L(A^{\KS}_s)$ is non-zero is indirect, and uses the validity of the first assertion for points valued in arbitrary orthogonal Shimura varieties. We direct the reader to Section \ref{sec:special} for details.

\subsection*{Moduli of K3 surfaces and the period map}
As mentioned above, a key component of this paper is a period map for K3 surfaces in odd characteristic. The classical Torelli map for K3 surfaces can be viewed as a map
\[
 \iota_{\Comp}:\tilde{\mathsf{M}}^{\circ}_{2d,\Comp}\to\Sh(L_d)_{\Comp},
\]
where $\mathsf{M}^{\circ}_{2d}$ is the moduli space (over $\Int[2^{-1}]$) of degree $2d$ primitively polarized K3 surfaces and $\tilde{\mathsf{M}}^{\circ}_{2d}$ is a certain $2$-fold `orientation' cover. $\Sh(L_d)$ is the associated orthogonal Shimura variety over $\Rat$.

Results of Rizov~\cite{rizov:cm} show that the period map descends over $\Rat$:
\[
 \iota_{\Rat}:\tilde{\mathsf{M}}^{\circ}_{2d,\Rat}\to\Sh(L_d).
\]

The following theorem is a positive characteristic analogue of the Torelli theorem for K3 surfaces.
\begin{introthm}\label{intro:thm:open}
There exists a regular integral model $\Ss(L_d)$ for $\Sh(L_d)$ over $\Int[2^{-1}]$ such that $\iota_{\Rat}$ extends to an \'etale map
\[
\iota_{\Int[2^{-1}]}:\tilde{\mathsf{M}}^{\circ}_{2d,\Int[2^{-1}]}\to\Ss(L_d).
\]
\end{introthm}
Over $\Int[(2d)^{-1}]$, this construction of the map is essentially due to Rizov~\cite{rizov:kugasatake}; cf. also \cite{maulik}*{\S 5}. With the same condition on $p$, a construction by Vasiu can be found in \cite{vasiu:k3}.
\comment{
To prove Theorem \ref{intro:thm:open}, one needs the \emph{integral} crystalline compatibility of the Kuga-Satake construction. In \cite{maulik}, this is shown using Fontaine-Laffaille theory, which is partly responsible for the assumption $p\geq 5$ there, as well as in \cite{charles}. Here, we present a proof that also works for $p=3$, exploiting the de Rham property of absolute Hodge cycles on abelian varieties, guaranteed us by a theorem of Blasius-Wintenberger~\cite{blasius}. We first prove integral compatibility directly for ordinary K3 surfaces by appealing to an integral Hodge-Tate comparison isomorphism of Bloch-Kato~\cite{bloch_kato}. We then use the density of the ordinary locus to propagate integrality to all of $\mathsf{M}^{\circ}_{2d,K,\overline{\Field}_p}$.
}
As a consequence of Theorem~\ref{intro:thm:open}, we get:
\begin{introcor}\label{intro:cor:moduli}
For any prime $p>2$, the moduli stack $\mathsf{M}^{\circ}_{2d,\Field_p}$ is quasi-projective. If $p^2\nmid d$, then $\mathsf{M}^{\circ}_{2d,\Field_p}$ is geometrically irreducible.
\end{introcor}
The quasi-projectivity was also proven in \cite{maulik}*{\S 5} for $p\geq 5$ with $p\nmid d$.

\subsection*{Further remarks}
There remains the question of extending these results to characteristic $2$. A major hindrance is the lack of a good theory of integral models of orthogonal Shimura varieties over $2$-adic rings of integers; cf.~\cite{mp:toroidal}*{4.6.5} for a discussion. Once such a theory is available, it should be straightforward to extend the ideas here to the situation where $2\nmid d$, though highly $2$-divisible $d$ are likely to present new difficulties.

In characteristic $0$, it is known that the period map is surjective, once extended to the moduli of quasi-polarized K3 surfaces. We expect the same assertion to hold in characteristic $p$. This question is intimately related that of the existence of a Ner\'on-Ogg-Shafarevich type criterion for the good reduction of K3 surfaces over discrete valuation fields of characteristic $p$. Such a criterion is available in characteristic $0$~\cites{kulikov,persson_pinkham}, and for certain K3 surfaces in finite characteristic~\cite{matsumoto}.

The Kuga-Satake construction has appeared in many other contexts in characteristic $0$: cf.~\cites{voisin:cubic,rapoport:appendix,andre:shaf,lyons}. It is likely that the methods of this paper will permit us to extend the construction into positive characteristic in these cases as well, enabling us to also prove the Tate conjecture in these contexts. Certainly, for cubic fourfolds, the Torelli theorem from \cite{voisin:cubic} allows us to apply our methods in rather straightforward fashion, and we indicate this briefly in (\ref{special:subsec:cubicfourfolds}; cf. also~\cite{levin} and \cite{charles}*{Corollary 6}.

\comment{
\subsection*{Tour of contents}
We begin in Section~\ref{sec:hodge} with a review of the theory of motives attached to absolute Hodge cycles, since this gives us a very powerful framework in which to place the Kuga-Satake correspondence. In particular, it permits us to show its compatibility with cohomological realizations in a rather natural way.

Section~\ref{sec:k3} is a quick review of the theory of moduli of (quasi-)polarized K3 surfaces, and in Section~\ref{sec:ortho}, we review what we need from \cite{mp:reg} about Shimura varieties of Spin and orthogonal type and their integral models.

In Section~\ref{sec:ks}, we use results from the preceding sections to extend the Kuga-Satake map over $\Int_{(p)}$, and

Section~\ref{sec:special} is the heart of the paper. Here, we prove the Tate conjecture for special endomorphisms, and }

\subsection*{Notational conventions}
For any prime $\ell$, $\nu_{\ell}$ will be the $\ell$-adic valuation satisfying $\nu_{\ell}(\ell)=1$. $\Adele_f$ will denote the ring of finite ad\'eles over $\Rat$, and $\widehat{\Int}\subset\Adele_f$ will be the pro-finite completion of $\Int$. Given a rational prime $p$, $\Adele_f^p$ will denote the ring of prime-to-$p$ finite ad\'eles; that is, the restricted product $\prod'_{\ell\neq p}\Rat_{\ell}$. Moreover, $\widehat{\Int}^p\subset\Adele_f^p$ will be the closure of $\Int$. Given a perfect field $k$ of finite characteristic, $W(k)$ will denote its ring of Witt vectors, and $\sigma:W(k)\to W(k)$ will be the canonical lift of the Frobenius automorphism of $k$. For any group $G$, $\underline{G}$ will denote the locally constant \'etale sheaf (over a base that will be clear from context) with values in $G$.

\subsection*{Acknowledgements}
We thank Anand Deopurkar, Mark Kisin, Davesh Maulik, George Pappas, Peter Scholze and Junecue Suh for helpful comments and conversations. We also thank an anonymous referee for some clarifying remarks. This work was partially supported by NSF Postdoctoral Research Fellowship DMS-1204165 and an AMS Simons Travel Grant.

\setcounter{tocdepth}{1}
\tableofcontents

\section{Motives}\label{sec:hodge}

Throughout this section (and only here), all fields will be assumed to be embeddable in $\Comp$, and all varieties will be smooth, projective. Our main reference for this section is \cite{dmos}.

\subsection{}
Given a field $k$ in characteristic $0$, denote by $\mb{Mot}_{\on{AH}}(k)$ the neutral $\Rat$-linear Tannakian category of motives over $k$ for absolute Hodge cycles; cf.~\cite{panchishikin}*{\S~2}, where it is denoted $\mathscr{M}_k$. Its objects are triples $M=(X,n,\varpi)$, where $X$ is a smooth projective variety over $k$, $n\in\Int$, and $\varpi$ is an idempotent absolutely Hodge self-correspondence of $X$. Given such an $M$ and $m\in\Int$, we will write $M(m)$ for the Tate twist $(X,n+m,\varpi)$. Write $h(X)$ for the motive $(X,0,\on{id})$.

For each embedding $\sigma:k\into\Comp$, Betti cohomology gives us a realization functor $\omega_{\sigma}$ for $\mb{Mot}_{\on{AH}}(k)$ into $\Rat$-vector spaces. For each prime $\ell$, $\ell$-adic cohomology gives us a realization functor $\omega_{\ell}$ into $\Rat_{\ell}$-vector spaces.\footnote{One also needs an additional choice of an algebraically closed field containing $k$, which we suppress.} In fact, the varying $\ell$-adic cohomology theories can be put together to obtain a realization functor $\omega_{\Adele_f}$ into $\Adele_f$-vector spaces. Finally, de Rham cohomology gives us a realization functor $\omega_{\dR}$ into $k$-vector spaces. For $?=\sigma,\ell,\Adele_f,\dR$, and $M\in\mb{Mot}_{\on{AH}}(k)$, we will write $M_{?}$ for the realization $\omega_{?}(M)$, especially when we want to call attention to additional structure: that of a Hodge structure, Galois-module, or filtered vector space, respectively.

For each variety $X$, the K\"unneth decomposition on $X\times X$ allows us to attach to each $d\in\Int_{\geq 0}$, an object $h^d(X)\in\mb{Mot}_{\on{AH}}(k)$ such that $\omega_{?}(h^d(X))=H^d_{?}(X)$, for $?=\sigma,\ell,\Adele_f,\dR$. If $H\in\on{CH}^1(X)$ is a hyperplane section, then the Lefschetz decomposition gives us an object $p^d(X)\in\mb{Mot}^+{\on{AH}}(k)$ such that $\omega_{?}(p^d(X))=PH^d_{?}(X)$, the primitive cohomology group associated with $H$; cf.~\cite{dmos}*{\S II.6}.

The following result is shown in \cite{dmos}*{II.6.7}.
\begin{prp}\label{hodge:prp:basechange}
For any extension $L/k$, there is a natural, faithful functor of Tannakian categories compatible with fiber functors:
  \[
   \_\otimes_kL: \mb{Mot}_{\on{AH}}(k)\rightarrow\mb{Mot}_{\on{AH}}(L).
  \]
  If $k$ is algebraically closed in $L$, then this functor is also full. In general, for motives $M,N\in\mb{Mot}_{\on{AH}}(k)$, a map $f:M\otimes_kL\rightarrow N\otimes_kL$ is defined over $k$ if and only if, for some prime $\ell$, its $\ell$-adic realization $f_{\ell}$ commutes with $\Aut(L/k)$.
\end{prp}
\qed

The following can be easily deduced from the main result of~\cite{dmos}*{Ch. I}.
\begin{thm}[Deligne]\label{hodge:thm:abshodge}
Let $\mb{Mot}_{\on{Ab}}(k)\subset\mb{Mot}_{\on{AH}}(k)$ be the full Tannakian sub-category generated by the motives attached to abelian varieties. Let $\mb{Hdg}_{\Rat}$ be the Tannakian category of $\Rat$-Hodge structures. Then, for any embedding $\sigma:k\into\Comp$, the functor
\begin{align*}
 \mb{Mot}_{\on{Ab}}(k)&\rightarrow\mb{Hdg}_{\Rat}\\
 M&\mapsto M_{\sigma}
\end{align*}
is faithful. If $k$ is algebraically closed, then it is in fact fully faithful.
\end{thm}
\qed

\subsection{}\label{hodge:subsec:rstruc}
We will need a mildly refined notion of a motive: Let $R\subset\Rat$ be a sub-ring. A \defnword{motive with $R$-structure} or an \defnword{$R$-motive} is a motive $M$ equipped with an $\Aut(\overline{k}/k)$-stable $R\otimes\widehat{\Int}$-lattice $M_{\widehat{R}}\subset M_{\Adele_f}$. Here, we write $M_{\Adele_f}$ for the . For example, if $R=\Int$, then $M_{\widehat{R}}$ is a $\widehat{\Int}$-lattice; and, if $R=\Int_{(p)}$, then giving $M_{\widehat{R}}$ amounts to giving a $\Aut(\overline{k}/k)$-stable $\Int_p$-lattice $M_{\Int_p}\subset M_p$.

A morphism $f:(M,M_{\widehat{R}})\to (N,N_{\widehat{R}})$ of $R$-motives is a map $f:M\to N$ of motives such that the $\Adele_f$-realization $f_{\Adele_f}$ carries $M_{\widehat{R}}$ into $N_{\widehat{R}}$.

Suppose that $M_{R}=(M,M_{\widehat{R}})$ is an $R$-motive. For any embedding $\sigma:k\into\Comp$, this also gives us a canonical $R$-lattice $M_{R,\sigma}\subset M_{\sigma}$ obtained as follows. Choose an extension $\overline{\sigma}:\overline{k}\into\Comp$ of $\sigma$. This gives us a comparison isomorphism
\[
 M_{\sigma}\otimes\Adele_f\xrightarrow{\simeq}M_{\Adele_f}.
\]
We now take $M_{R,\sigma}$ to be the intersection of the pre-image of $M_{\widehat{R}}$ with $M_{\sigma}$. Since $M_{\widehat{R}}$ is $\Aut(\overline{k}/k)$-stable, this does not depend on the choice of $\overline{\sigma}$. Clearly, for any map $f:M_{R}\to N_{R}$ of $R$-motives, the Betti realization $f_\sigma$ respects the $R$-lattice $M_{R,\sigma}$.

Given an $R$-motive $M$, and a prime $p$ not invertible in $R$, we will write $M_p$ for its associated $\Int_p$-representation of $\Aut(\overline{k}/k)$, and, for any $\sigma:k\into\Comp$, we will write $M_{\sigma}$ for the associated $R$-Hodge structure.\comment{Note that, for each variety, $X$, $h(X)$ has a natural $R$-structure. In particular, the Lefschetz motive $L$ underlies a natural $R$-motive which we will continue to denote by $L$.}

For any $R$-motive $M$, write $\on{AH}(M)$ for the $R$-module of \defnword{cycles on $M$}: This is the space of maps $\Hom(\mb{1},M)$, where $\mb{1}$ is the identity object; that is $\mb{1}=h(\on{pt})$ with its natural $R$-structure. If $R\into R'$ is an inclusion of sub-rings of $\Rat$, then there is a natural functor $\_\otimes_RR'$ from $R$-motives to $R'$-motives such that
\[
 \on{AH}(M)\otimes_RR'=\on{AH}(M\otimes_RR'),
\]
for any $R$-motive $M$.

\begin{defn}\label{hodge:defn:polarization}
  An $R$-motive $M$ is \defnword{pure of weight $d$}, for some $d\in\Int$, if, for one (hence all) $\sigma:k\into\Comp$, $M_{\sigma}$ is a pure Hodge structure of weight $d$. A \defnword{polarization} on a $R$-motive $M$ that is pure of weight $d$ is a pairing
  \[
   \psi:M\otimes M\rightarrow L(-d)
  \]
  such that, for any $\sigma:k\into\Comp$, $\psi$ induces a polarization of the $\Rat$-Hodge structure $M_{\sigma}\otimes_R\Rat$.
\end{defn}

\subsection{}\label{hodge:subsec:admotives}
One problem with absolute Hodge cycles is that they do not have an analogue in positive characteristic. We will deal with this in somewhat \emph{ad hoc} fashion. We now assume that the field $k$ is equipped with a discrete valuation $\nu:k^\times\to\Int$ such that the residue field $k(\nu)$ is perfect of characteristic $p>0$. Let $k_{\nu}$ be the completion of $k$ along $\nu$, and let $\Reg{\nu}$ be its ring of integers. Let $\Bdr$ be Fontaine's ring of de Rham periods for $k_{\nu}$. For any smooth projective variety over $k$, and for $d\in\Int_{\geq 0}$, we have the de Rham comparison isomorphism:
\begin{align*}
  \gamma_{\dR}:H^d_p(X)\otimes_{\Rat_p}\Bdr&\xrightarrow{\simeq}H^d_{\dR}(X/k)\otimes_k\Bdr.
\end{align*}
Here, we write $H^d_p(X)$ for the $p$-adic cohomology group $H^d_{\et}(X_{\overline{k}},\Rat_p)$, where $\overline{k}$ is an algebraic closure of $k$.

\begin{defn}\label{hodge:defn:absderham}
An absolutely Hodge cycle $s$ on $X$ with $p$-adic realization $s_p$ and de Rham realization $s_{\dR}$ is \defnword{de Rham} (with respect to $\nu$), if
\[
 \gamma_{\dR}(s_p\otimes 1)=s_{\dR}\otimes 1.
\]
\end{defn}

Let $\mb{Mot}_{\on{AD},\nu}(k)$ be the category defined exactly as $\mb{Mot}_{\on{AH}}(k)$ is in \cite{panchishikin}*{\S~2}, except that we only allow absolutely de Rham cycles as morphisms. It is easy to see that this is a sub-category of $\mb{Mot}_{\on{AH}}(k)$. The analogue of \cite{dmos}*{II.6.2} holds in this setting, so $\mb{Mot}_{\on{AD},\nu}(k)$ is semi-simple and in fact Tannakian.

\begin{thm}[Blasius-Wintenberger]\label{hodge:thm:blasius}
Let $\mb{Mot}_{\on{Ab},\nu}(k)$ be the Tannakian sub-category of $\mb{Mot}_{\on{AD},\nu}(k)$ generated by the motives attached to abelian varieties. Then the natural functor
\[
 \mb{Mot}_{\on{Ab},\nu}(k)\to\mb{Mot}_{\on{Ab}}(k)
\]
is an equivalence of categories.
\end{thm}
\begin{proof}
  This reduces to showing that every (absolutely) Hodge cycle on an abelian variety is de Rham, which is the main result of \cite{blasius}.
\end{proof}

\subsection{}
We will now work with pairs $(X,\mf{X})$, where $X$ is a $k$-variety and $\mf{X}$ is a smooth proper $\Reg{\nu}$-scheme equipped with an identification $\mf{X}\otimes_{\Reg{\nu}}k(\nu)=X\otimes_kk(\nu)$. Write $\mf{X}_0$ for the special fiber $\mf{X}\otimes_{\Reg{\nu}}k(\nu)$. Set $W=W(k(\nu))$; then the crystalline cohomology $H^d_{\cris}(\mf{X}_0/W)$ is an $F$-crystal over $W$.

Let $W(-1)=H^2_{\cris}(\bb{P}^1_{k(v)}/W)$, and let $W(1)$ be its dual; note that $W(1)\left[\frac{1}{p}\right]$ has the structure of an $F$-isocrystal over $W\left[\frac{1}{p}\right]$, but that $W(1)$ is not $F$-stable. 

Let $\Bcris$ be Fontaine's ring of crystalline periods for $k_{\nu}$. For $d\in\Int_{\geq 0}$ and $m\in\Int$, we have natural comparison isomorphisms:
\begin{align*}
 \gamma_{\on{B-O}}:H^d_{\cris}(\mf{X}_0/W)(m)\otimes_Wk_{\nu}&\xrightarrow{\simeq}H^d_{\dR}(X)(m)\otimes_kk_{\nu};\\
 \gamma_{\cris}: H^d_p(X)(m)\otimes_{\Rat_p}\Bcris&\xrightarrow{\simeq}H^d_{\cris}(\mf{X}_0/W)(m)\otimes_W\Bcris.
\end{align*}
These isomorphisms are compatible in the sense that
\[
 \gamma_{\on{B-O}}\circ(\gamma_{\cris}\otimes 1)=\gamma_{\dR}.
\]

\begin{defn}\label{hodge:defn:abscrys}
An absolutely Hodge cycle $s\in H^d_{\Adele_f}(X)(m)\times H^d_{\dR}(X)(m)$ is \defnword{Tate} (with respect to $\mf{X}$ and $\nu$), if
\[
\gamma_{\on{B-O}}^{-1}(s_{\dR}\otimes 1)\in H^d_{\cris}(\mf{X}_0/W)(m)\otimes_Wk(\nu)
\]
is an $F$-invariant element of $H^d_{\cris}(\mf{X}_0/W)(m)\left[\frac{1}{p}\right]$. We will denote this $F$-invariant element by $s_{\cris}$: it is the \emph{crystalline} realization of $s$.

We say that $s$ is \defnword{crystalline} (with respect to $\mf{X}$ and $\nu$) if it is Tate, and if $\gamma_{\cris}(s_p\otimes 1)=s_{\cris}\otimes 1$.

Since the comparison isomorphisms are compatible with cycle classes and Poincar\'e duality, we see that algebraic cycle classes are crystalline. Similar statements hold for the K\"unneth and Lefschetz decompositions.
\end{defn}

\begin{lem}\label{hodge:lem:abscrys}
The notion of being Tate or crystalline does not depend on the choice of model $\mf{X}$. In fact, an absolutely Hodge cycle is crystalline if and only if it is de Rham. Moreover, the $F$-isocrystal $H^d_{\cris}(\mf{X}_0/W)\left[\frac{1}{p}\right]$ is also independent of the choice of model $\mf{X}$.
\end{lem}
\begin{proof}
  Since $\gamma_{\dR}$ is compatible with $\gamma_{\on{B-O}}$ and $\gamma_{\cris}$, and since $\gamma_{\cris}(s_p\otimes 1)$ is always $F$-invariant, $s$ is crystalline if and only if $\gamma_{\dR}(s_p\otimes 1)=s_{\dR}\otimes 1$; that is, if and only if $s$ is de Rham. From this, the first two assertions are immediate.

  For the third, we now only have to note that $H^d_{\cris}(\mf{X}_0/W)\left[\frac{1}{p}\right]$ is identified with the $\Gal(\overline{k}_{\nu}/k_{\nu})$-invariants of $H^d_p(X)\otimes_{\Rat_p}\Bcris$.
\end{proof}
\comment{
From (\ref{hodge:thm:blasius}), we now have:
\begin{corollary}\label{hodge:cor:abvarcrys}
If $X$ is an abelian variety over $k$ with good reduction at $\nu$, then we have:
\[
 \on{AC}^{r,s}(X(m))=\on{AD}^{r,s}(X(m))=\on{AH}^{r,s}(X(m)).
\]
\end{corollary}
\qed
}
\subsection{}\label{hodge:subsec:acah}
Let $\mb{Mot}_{\on{AC},\nu}(k)\subset \mb{Mot}_{\on{AH}}(k)$ be the sub-category whose objects are triples $(X,m,\pi)$, where $X$ has good reduction at $\nu$, and $\pi$ is crystalline. Morphisms are given as before, except that we restrict ourselves to absolutely crystalline cycles. Just like $\mb{Mot}_{\on{AD},\nu}(k)$,  $\mb{Mot}_{\on{AC},\nu}(k)$ is also Tannakian. Note that, by (\ref{hodge:lem:abscrys}) above, any object $M$ of $\mb{Mot}_{\on{AC},\nu}$ has a canonical crystalline realization $M_{\cris}$ that is an $F$-isocrystal over $W\left[\frac{1}{p}\right]$ and is equipped with a natural isomorphism of $k_{\nu}$-vector spaces
\[
 M_{\cris}\otimes_{W\left[\frac{1}{p}\right]}k_{\nu}\xrightarrow{\simeq}M_{\dR}\otimes_kk_{\nu}.
\]

The next result follows easily from (\ref{hodge:lem:abscrys}) and (\ref{hodge:thm:blasius}):
\begin{prp}\label{hodge:prp:abelian}
Let $\mb{Mot}^{\circ}_{\on{Ab},\nu}(k)$ (resp. $\mb{Mot}^{\circ}_{\on{Ab},\nu,\cris}(k)$) be the full sub-category of $\mb{Mot}_{\on{AH}}(k)$ (resp. $\mb{Mot}^{\circ}_{\on{AC},\nu}(k)$) generated by the motives attached to abelian varieties with good reduction at $\nu$. Then the natural functor
\[
 \mb{Mot}^{\circ}_{\on{Ab},\nu,\cris}(k)\to\mb{Mot}_{\on{Ab}}(k)
\]
is fully faithful and its essential image is $\mb{Mot}^{\circ}_{\on{Ab},\nu}(k)$.
\end{prp}
\qed

\section{Moduli of K3 surfaces}\label{sec:k3}

Our main references for this section are \cites{rizov:moduli,rizov:cm,maulik,ogus:ss}.

\subsection{}
A \defnword{K3 surface} over a scheme $S$ is an algebraic space $f:X\to S$ over $S$ that is proper, smooth and whose geometric fibers are K3 surfaces. A \defnword{polarization} (resp. a \defnword{quasi-polarization}) of a K3 surface $X\to S$ is a section $\xi\in\underline{\Pic}(X/S)(S)$ whose fiber at each geometric point $s\to S$ is a polarization (resp. a quasi-polarization); that is, the class of an ample (resp. big and nef\footnote{`big' equals being the tensor product of an ample line bundle with an effective one.}) line bundle, of the K3 surface $X_s$ over $k(s)$. There is an intersection pairing on $\underline{\Pic}(X/S)$ with values in the locally constant sheaf $\underline{\Int}$; the \defnword{degree} $\deg(\xi)\in H^0(S,\underline{\Int})$ of a (quasi-)polarization $\xi$ is the value of its pairing with itself. The restriction of $\deg(\xi)$ to any connected component of $S$ is a non-zero positive integer. A section $\xi$ of $\underline{\Pic}(X/S)$ is \defnword{primitive} if, for all geometric points $s\to S$, $\xi(s)$ is primitive; that is, $\xi(s)$ is not a non-trivial multiple of any element of $\Pic(X_s)$.

Fix an integer $d\in\Int_{>0}$, and let $\mathsf{M}_{2d}$ (resp. $\mathsf{M}^{\circ}_{2d}$) be the moduli problem over $\Int\left[\frac{1}{2}\right]$ that assigns to every $\Int\left[\frac{1}{2}\right]$-scheme $S$ the groupoid of tuples $(f:X\to S,\xi)$, where $X\to S$ is a K3 surface and $\xi$ is a primitive quasi-polarization (resp. polarization) of $X$ with $\deg(\xi)=2d$.
\begin{prp}\label{k3:prp:representability}
The natural map $\mathsf{M}^{\circ}_{2d}\to\mathsf{M}_{2d}$ is an open immersion of Deligne-Mumford stacks of finite type over $\Int$, fiber-by-fiber dense. Moreover, $\mathsf{M}^{\circ}_{2d}$ is separated.
\end{prp}
\begin{proof}
  Everything except the fiber-by-fiber density of the image of the map can be found in \cite{rizov:moduli}*{4.3.3} and \cite{maulik}*{Proposition 2.1}. Showing the claimed density amounts to seeing that any quasi-polarized K3 surface $(X_0,\xi_0)$ over a field $k$ admits a deformation $(X,\xi)$ such that $\xi$ is an ample class. Indeed, let $D_0$ be a divisor on $X_0$ with class $\xi_0$. Then $3D_0$ determines a base-point free map $X_0\to\bb{P}^N$ whose image is a surface with isolated ordinary double-point singularities. The pre-images of the singularities are $(-2)$-rational curves on $X_0$. If a deformation $(X,\xi)$ of $(X_0,\xi_0)$ is not polarized, then one of these $(-2)$-curves must also permit a deformation to $X$. It is easy to check using the Riemann-Roch formula that deforming a $(-2)$-curve on a K3 surface is equivalent to deforming its divisor class, and so \cite{lieblich_olsson}*{Theorem A.7} shows that the deformation locus of a $(-2)$-curve in the versal deformation space of $(X_0,\xi_0)$ has co-dimension $1$. This implies in turn that the locus where the versal deformation is not polarized is a union of co-dimension $1$ sub-spaces, and so finishes the proof of the proposition. Notice that the proof shows that the complement of $\mathsf{M}^{\circ}_{2d}$ in $\mathsf{M}_{2d}$ is flat over $\Int\left[\frac{1}{2}\right]$ and has pure co-dimension $1$.
\end{proof}

\subsection{}\label{k3:subsec:elladic}
Let $(\bm{f}:\bm{\mathcal{X}}\to\mathsf{M}_{2d},\bm{\xi})$ be the universal object over $\mathsf{M}_{2d}$. For any prime $\ell$, the second relative \'etale cohomology $\bm{H}^2_{\ell}$ of $\bm{\mathcal{X}}$ over $\mathsf{M}_{2d,\Int\left[\frac{1}{2\ell}\right]}$ with coefficients in $\underline{\Int}_{\ell}$ is a lisse $\Int_{\ell}$-sheaf of rank $22$ equipped with a perfect, symmetric Poincar\'e pairing
\[
 \langle\_,\_\rangle: \bm{H}^2_{\ell}\times\bm{H}^2_{\ell}\rightarrow\underline{\Int}_{\ell}(-2).
\]
We will actually be equipping $\bm{H}^2_{\ell}$ with the \emph{negative} of the conventional pairing. In characteristic $0$, this means that we are viewing the Betti cohomology groups of K3 surfaces as being quadratic spaces of signature $(19+,3-)$.

The $\ell$-adic Chern class $\on{ch}_{\ell}(\bm{\xi})$ of $\bm{\xi}$ is a global section of the Tate twist $\bm{H}^2_{\ell}(1)$ that satisfies $\langle \on{ch}_{\ell}(\xi),\on{ch}_{\ell}(\xi)\rangle=-2d$. We set
\[
 \bm{P}^2_{\ell}=\langle \on{ch}_{\ell}(\bm{\xi})\rangle^{\perp}(-1)\subset\bm{H}^2_{\ell}.
\]
This is a lisse $\Int_{\ell}$-sheaf over $\mathsf{M}_{2d,\left[\frac{1}{2\ell}\right]}$ of rank $21$ and it inherits a symmetric $\Int_{\ell}(-2)$-valued pairing $\langle\_,\_\rangle$, which is perfect if $\ell\nmid d$.

\subsection{}
There is also the second relative de Rham cohomology $\bm{H}^2_{\dR}$ of $\bm{\mathcal{X}}$ over $\mathsf{M}_{2d}$. This is a vector bundle with flat connection of rank $22$ equipped with a Hodge filtration $F^\bullet\bm{H}^2_{\dR}$ satisfying Griffiths transversality.  It is also equipped with a perfect, horizontal, symmetric pairing $\langle\_,\_\rangle$ into $\Reg{\mathsf{M}_{2d}}$. The filtration then is of the form
\[
 0=F^3\bm{H}^2_{\dR}\subset F^2\bm{H}^2_{\dR}\subset F^1\bm{H}^2_{\dR}=\bigl(F^2\bm{H}^2_{\dR}\bigr)^{\perp}\subset F^{0}\bm{H}^2_{\dR}=\bm{H}^2_{\dR},
\]
determined by the isotropic line $F^2\bm{H}^2_{\dR}$. The de Rham Chern class $\on{ch}_{\dR}(\bm{\xi})$ attached to $\bm{\xi}$ is a horizontal global section of $F^1\bm{H}^2_{\dR}$ satisfying $\langle \on{ch}_{\dR}(\bm{\xi}),\on{ch}_{\dR}(\bm{\xi})\rangle=-2d$. Again, we set
\[
 \bm{P}^2_{\dR}=\langle \on{ch}_{\dR}(\bm{\xi})\rangle^{\perp}\subset\bm{H}^2_{\dR}.
\]
This is a vector sub-bundle of $\bm{H}^2_{\dR}$ of rank $21$, and it inherits the connection, the filtration and the symmetric pairing from $\bm{H}^2_{\dR}$.

For any prime $p$, over $\mathsf{M}_{2d,\Field_p}$, the induced vector bundle $\bm{H}^2_{\dR,\Field_p}$ is equipped with an decreasing, horizontal filtration $F_{\on{con}}^\bullet\bm{H}^2_{\dR,\Field_p}$ called the \defnword{conjugate filtration} (cf.~\cite{ogus:ss}*{\S 1} for this and the rest of the discussion in this paragraph). Suppose that $k$ is an algebraically closed field over $\Field_p$ and we have a map $s:\Spec k\to\mathsf{M}_{2d}$. We say that $s$ is \defnword{superspecial} if the fiber of $\on{ch}_{\dR}(\bm{\xi})$ in $F^1\bm{H}^2_{\dR,s}$ lies in $F^2\bm{H}^2_{\dR,s}$. In this case, we have
\[
  F^2\bm{H}^2_{\dR,s}=F^2_{\on{con}}\bm{H}^2_{\dR,s}.
\]
We say that $s$ is \defnword{ordinary} if $\bm{X}_s$ is ordinary; that is, if $F^2\bm{H}^2_{\dR,s}\cap F^2_{\on{con}}\bm{H}^2_{\dR,s}=0$.

We now recall some definitions and results from \cite{vasiu:zink}.
\begin{defn}\label{k3:defn:healthy}
A regular local $\Int_{(p)}$-algebra $R$ with maximal ideal $\mx$ is \defnword{quasi-healthy} if it is faithfully flat over $\Int_{(p)}$, and if every abelian scheme over $\Spec~R\backslash\{\mx\}$ extends uniquely to an abelian scheme over $\Spec~R$.

A regular $\Int_{(p)}$-scheme $X$ is \defnword{healthy} if it is faithfully flat over $\Int_{(p)}$, and if, for every open sub-scheme $U\subset X$ containing $X_{\Rat}$ and all generic points of $X_{\Field_p}$, every abelian scheme over $U$ extends uniquely to an abelian scheme over $X$. It is \defnword{locally healthy} if, for every point $x\in X_{\Field_p}$ of co-dimension at least $2$, the complete local ring $\widehat{\Rg}_{X,x}$ is quasi-healthy.
\end{defn}

\begin{rem}\label{k3:rem:healthy}
\begin{itemize}
\item Any regular, flat $\Int_{(p)}$-scheme of dimension at most $1$ is trivially healthy.

\item By faithfully flat descent, a regular local ring $R$ is quasi-healthy whenever its completion $\widehat{R}$ is quasi-healthy.

\item If $X$ is locally healthy, then it is healthy. Indeed, suppose that $U\subset X$ is as in the definition of `healthy' above; the complement $X\backslash U$ lies entirely in the special fiber and has co-dimension at least $2$ in $X$. The claim follows by using ascending Noetherian induction on the co-dimension of $X\backslash U$, and repeatedly using quasi-healthiness of the local rings of $X$.

    We do not know if the converse holds.
\end{itemize}
\end{rem}

\begin{thm}[(Vasiu-Zink)]\label{k3:thm:vasiuzink}
Let $R$ be a regular local, faithfully flat $\Int_{(p)}$-algebra of dimension at least $2$.
\begin{enumerate}
\item\label{vz:main}Suppose that there exists a faithfully flat complete local $R$-algebra $\hat{R}$ that admits a surjection $\hat{R}\twoheadrightarrow W\pow{T_1,T_2}/(p-h)$, where $h\in (T_1,T_2)W\pow{T_1,T_2}$ is a power series that does not belong to the ideal $(p,T_1^p,T_2^p,T_1^{p-1}T_2^{p-1})$. Then $R$ is quasi-healthy.
\item\label{vz:easy}Let $\mx_R\subset R$ be the maximal ideal and suppose that $p\notin\mx_R^p$. Then $R$ is quasi-healthy.
\item\label{vz:smooth}If $R$ is a formally smooth complete local $\Int_{(p)}$-algebra, then $R$ is quasi-healthy.
\end{enumerate}
\end{thm}
\begin{proof}
  Cf.~Theorem 3 and Corollary 4 of~\cite{vasiu:zink}.
\end{proof}

We can encapsulate the deformation theory of K3 surfaces in the following
\begin{thm}\label{k3:thm:deformation}
Let $X_0$ be a K3 surface over a perfect field $k$ of characteristic $p>0$. Then:
\begin{enumerate}
  \item\label{deform:smooth}The deformation functor $\Def_{X_0}$ for $X_0$ is pro-representable and formally smooth of dimension $20$ over $W(k)$.
  \item\label{deform:xi0}For any class $\xi_0\in\Pic(X_0)$, the deformation functor $\Def_{(X_0,\xi_0)}$ for the pair $(X_0,\xi_0)$ is pro-represented by a flat, formal sub-scheme of $\Def_{X_0}$ defined by a single equation.
  \item\label{deform:primitive}If $\xi_0$ is primitive, then $\on{ch}_{\dR}(\xi_0)\neq 0$, and $\Def_{(X_0,\xi_0)}$ is formally smooth, unless $\on{ch}_{\dR}(\xi_0)$ lies in $F^2H^2_{\dR}(X_0/k)$. In particular, $\Def_{(X_0,\xi_0)}$ is formally smooth whenever $X_0$ is ordinary.
  \item\label{deform:superspecial}If $\xi_0$ is primitive and $\on{ch}_{\dR}(\xi_0)$ lies in $F^2H^2_{\dR}(X_0/k)$, then $\nu_p(\deg(\xi_0))=1$, and $\Def_{(X_0,\xi_0)}$ is quasi-healthy regular.
\end{enumerate}
\end{thm}
\begin{proof}
  (\ref{deform:smooth}) and (\ref{deform:xi0}) are due to Deligne; cf.~\cite{deligne:k3liftings}*{1.2,1.5}. (\ref{deform:primitive}) can be found in \cite{ogus:ss}*{2.2}.

  For (\ref{deform:superspecial}), that $\nu_p(\deg(\xi_0))=1$ follows from (a suitable adaptation of) the argument in \cite{mp:reg}*{5.20}. The main point is that the de Rham cohomology of any lift of $X_0$ over $W(k)$ is a strongly divisible filtered $F$-crystal in the sense of \cite{ogus:transversal}*{3.9}. This follows from \cite{berthelot_ogus}*{8.26} and \cite{laffaille}*{Prop 5.2}.\comment{set $W=W(k)$, and let $W_{\Rat}$ be its fraction field. Let $R$ be the formally smooth $W$-algebra pro-representing $\Def_{X_0}$. Choose any map $R\to W$: this gives rise to a formal lift $X/W$ of $X_0$. Let $H=H^2_{\dR}(X/W)$ be the de Rham cohomology of $X$. Via the identification of $H$ with the crystalline cohomology of $X_0$, if $\sigma:W\to W$ is the Frobenius lift, we have a Frobenius map:
  \[
   F:\sigma^*H\rightarrow H.
  \]
  Let $F^\bullet H\subset H$ be the Hodge filtration on $H$. Then the strong divisibility of the filtered $F$-crystal $H^2_{\dR}(X/W)$ (\ref{ks:prp:strongdiv}) shows that we have:
  \begin{align}\label{special:eqn:strongdiv}
    F\biggl(\sigma^*\bigl(p^{-2}F^2H+p^{-1}F^1H+H\bigr)\biggr)=H.
  \end{align}

  Let $f\coloneqq \on{ch}_{\cris}(\xi_0)\in H$ be the crystalline Chern class of $\xi_0$. By our hypothesis in (\ref{deform:superspecial}), we can write
  \[
   f=f_1+pf_2,
  \]
  where $f_1$ is a generator for $F^2H$, and $f_2\in H$.

  We have:
  \[
   pf_1+p^2f_2=pf=F(f)=F(f_1)+pF(f_2).
  \]
  Here, for any $v\in H$, we write $F(v)$ for $F(\sigma^*v)$.

  Since $f_1\in F^2H$, necessarily $F(f_1)\in p^2H$, and we see that we have:
  \[
   F(f_2)=\frac{1}{p}\bigl(pf_1+p^2f_2-F(f_1)\bigr)\in H\backslash pH.
  \]
  So (\ref{special:eqn:strongdiv}) implies that $f_2$ does not lie in $pH+F^1H$. In other words, the image of $f_2$ in $\gr^0_FH$ is a generator. This shows that $f_1\cdot f_2$ is a unit, and so
  \[
   \deg(\xi_0)=f\cdot f=2p(f_1\cdot f_2)
  \]
  is not divisible by $p^2$.}

  Now, Ogus~\cite{ogus:ss}*{2.2} shows that the deformation ring for $\Def_{(X_0,\xi_0)}$ is isomorphic to
  \[
   W\pow{t_1,\ldots,t_{10},u_1,\ldots,u_{10}}/(\sum_it_iu_i-\deg(\xi_0)).
  \]
  So it follows from Vasiu and Zink's criterion (\ref{k3:thm:vasiuzink}) that this ring is quasi-healthy regular.
\end{proof}

\comment{
One of the ingredients in the proof of Deligne cited above is a description of the tangent space of the deformation functor, which we extract here for later reference.
\begin{lem}\label{k3:lem:tangentspace}
In the situation of (\ref{deform:xi0}) above, we have a canonical identification
\[
\Def_{(X_0,\xi_0)}\bigl(k[\epsilon]\bigr)=\biggl\{\text{Isotropic lines $L\subset PH^2_{\dR}(X_0/k)\otimes k[\epsilon]$ lifting $F^2H^2_{\dR}(X_0/k)$}\biggr\}.
\]
Here, as usual $PH^2_{\dR}(X_0/k)=\langle \on{ch}_{\dR}(\xi_0)\rangle^{\perp}\subset H^2_{\dR}(X_0/k)$.
\end{lem}
\begin{proof}
  This is standard. We only note that, under this identification, each deformation $(X,\xi)$ over $k[\epsilon]$ is mapped to the isotropic line
  \[
  F^2H^2_{\dR}(X/k)\subset H^2_{\dR}(X/k)=H^2_{\dR}(X_0/k)\otimes_kk[\epsilon].
  \]
\end{proof}
}
\begin{corollary}\label{k3:cor:regular}
Let $r$ be the product of primes $\ell>2$ such that $\ell\mid d$, but $\ell^2\nmid d$.
\begin{enumerate}
\item\label{k3:smooth}$\mathsf{M}_{2d,\Int\left[\frac{1}{2r}\right]}$ is smooth over $\Int\left[\frac{1}{2r}\right]$ of relative dimension $19$.
\item\label{k3:sing}If $p\mid r$, then the singular locus of $\mathsf{M}_{2d,\Field_p}$ is at most $0$-dimensional, and lies within the superspecial locus.
\item\label{k3:reg}All mixed characteristic complete local rings of $\mathsf{M}_{2d,\Int\left[\frac{1}{2}\right]}$ of dimension at least $2$ are quasi-healthy regular.
\end{enumerate}
\end{corollary}
\begin{proof}
(\ref{k3:smooth}) is an immediate consequence of (\ref{deform:primitive}) and (\ref{deform:superspecial}) of (\ref{k3:thm:deformation}).

For (\ref{k3:sing}), we first note that the singular points of $\mathsf{M}_{2d,\Field_p}$ are all superspecial and that their complete local rings are quasi-healthy regular, by \emph{loc. cit.}. The assertion is now a consequence of the fact that there are no non-trivial infinitesimal families of quasi-polarized superspecial K3 surfaces (cf~\cite{ogus:ss}*{Remark 2.7}).

For (\ref{k3:reg}), we only need to worry about the complete local rings of $\mathsf{M}_{2d,\Int\left[\frac{1}{2}\right]}$ at points valued in fields of characteristic $p\mid r$. By (\ref{k3:sing}), the completions at the non-closed such points are formally smooth and hence quasi-healthy regular. The completions at the closed such points are quasi-healthy regular, as we have already observed.
\end{proof}

\subsection{}\label{k3:subsec:level}
We will need moduli spaces of K3 surfaces with level structure; cf.~\cite{rizov:moduli}*{\S 4}. Let $U$ be the hyperbolic lattice over $\Int$ of rank $2$; let $N$ be the self-dual lattice $U^{\oplus 3}\oplus E_8^{\oplus 2}$. Choose a basis $e,f$ for (say) the first copy of $U$ in $N$. Set
\[
  L_d=\langle e-df\rangle^{\perp}\subset N.
\]
This is a quadratic lattice over $\Int$ of discriminant $2d$; let $\dual{L}_d\subset V_d\coloneqq L_{d,\Rat}$ be its dual lattice. Set $G_d=\SO(V_d)$: it is a semi-simple algebraic group over $\Rat$.

Let $K\subset G_d(\Adele_f)$ be a compact open sub-group that stabilizes $L_{d,\widehat{\Int}}$ and acts trivially on $\dual{L}_d/L_d$. The maximal such sub-group is called the \defnword{discriminant kernel} of $L_{d,\widehat{\Int}}$. These compact opens are called \defnword{admissible} in \cite{rizov:moduli}. Strictly speaking, Rizov's definition of admissibility is the following: First, note that $G_d$ can be viewed as the sub-group of isometries of $V$ that fix $e-df$. Now, a compact open sub-group $K\subset G_d(\Adele_f)$ is admissible if every element of $K$, viewed as an isometry of $V_{\Adele_f}$, stabilizes $L_{\widehat{\Int}}$. That this is equivalent to our definition is shown in \cite{mp:reg}*{2.2}.

We will now fix an admissible compact open $K\subset G_d(\Adele_f)$ such that $K_p\subset G_d(\Rat_p)$ is the discriminant kernel of $L_{d,\Int_p}$.

Over $\mathsf{M}_{2d,\Int_{(p)}}$, the relative $\ell$-adic cohomology sheaves $\bm{H}^2_{\ell}$, for $\ell\neq p$, can be put together to get the $\widehat{\Int}^p$-sheaf $\bm{H}^2_{\widehat{\Int}^p}=\prod_{\ell\neq p}\bm{H}^2_{\ell}$. Then the Chern classes of $\bm{\xi}$ can also be put together to get the Chern class $\on{ch}_{\widehat{\Int}^p}(\bm{\xi})$ in $\bm{H}^2_{\widehat{\Int}^p}(1)$. Let $I^p$ be the \'etale sheaf over $\mathsf{M}_{2d,\Int_{(p)}}$, whose sections over any scheme $T\to\mathsf{M}_{2d,\Int_{(p)}}$ are given by
\[
 I^p(T)=\bigl\{\text{Isometries $\eta:L\otimes\underline{\widehat{\Int}}^p\xrightarrow{\simeq}\bm{H}^2_{\widehat{\Int}^p,T}(1)$ with $\eta(e-df)=\on{ch}_{\widehat{\Int}^p}(\bm{\xi})$}\bigr\}
\]
This has a natural right action via pre-composition by the constant sheaf of groups $K^p$. A section $[\eta]\in H^0(T,I^p/K^p)$ is called a \defnword{$K^p$-level structure} over $T$.

We define $\mathsf{M}_{2d,K,\Int_{(p)}}$ to be the relative moduli problem over $\mathsf{M}_{2d,\Int_{(p)}}$ that attaches to $T\to\mathsf{M}_{2d,\Int_{(p)}}$ the set of $K^p$-level structures over $T$.

\comment{
Over $\Rat$, we can define another moduli space: Let $I_p$ be the \'etale sheaf over $\mathsf{M}_{2d,K^p,\Rat}$, whose sections over any $T\to\mathsf{M}_{2d,K^p,\Rat}$ are given by:
\[
 I_p(T)=\bigl\{\text{Isometries $\eta:L\otimes\underline{\widehat{\Int}}_p\xrightarrow{\simeq}\bm{H}^2_{p,T}(1)$ with $\eta(e-df)=\on{ch}_{p}(\bm{\xi})$}\bigr\}
\]
A section $[\eta]\in H^0(T,I_p/K_p)$ is a \defnword{$K_p$-level structure} over $T$. We take $\mathsf{M}_{2d,K,\Rat}$ to be the relative moduli problem over $\mathsf{M}_{2d,K^p,\Rat}$ that parameterizes $K_p$-level structures. Note that $\mathsf{M}_{2d,K,\Rat}$ is the stack denoted in \cite{rizov:moduli}*{6.2.2} as $\mathcal{F}_{2d,K,\Rat}^{\on{full}}$, and that, if $p\nmid d$, it is trivial over $\mathsf{M}_{2d,K^p,\Rat}$. Even in general, it has an easy description: One finds that giving a $K_p$-level structure is the same as giving a section of the locally constant sheaf
\[
 \frac{\bm{H}^2_p(1)}{\bm{P}^2_p(1)\oplus\langle \on{ch}_{p}(\bm{\xi})\rangle}\xrightarrow{\simeq}\frac{\dual{\langle \on{ch}_{p}(\bm{\xi})\rangle}}{\langle \on{ch}_{p}(\bm{\xi})\rangle}.
\]
So $\mathsf{M}_{2d,K,\Rat}$ admits a section over $\mathsf{M}_{2d,K^p,\Rat}$, and is therefore the trival two-sheeted cover. In particular, it extends uniquely to a finite \'etale cover $\mathsf{M}_{2d,K,\Int_{(p)}}\to\mathsf{M}_{2d,K^p,\Int_{(p)}}$.
}
\begin{prp}\label{k3:prp:levelrep}
$\mathsf{M}_{2d,K,\Int_{(p)}}$ is finite and \'etale over $\mathsf{M}_{2d,\Int_{(p)}}$. For $K^p$ small enough, it is an algebraic space over $\Int_{(p)}$. It is healthy regular, and, unless $\nu_p(d)=1$, it is smooth over $\Int_{(p)}$.
\end{prp}
\begin{proof}
  Both finiteness and \'etaleness are clear from the definition. As for the second assertion, the key point is to show that a quasi-polarized K3 surface with $K^p$-level structure has trivial automorphism group. This is shown in \cite{maulik}*{2.8}, which is based on \cite{rizov:moduli}*{6.2.2}.

  The last assertion follows from (\ref{k3:cor:regular}).
\end{proof}

\section{Shimura varieties}\label{sec:ortho}

Our main reference for this section will be \cite{mp:reg}.

\subsection{}
Let $L$ be a quadratic lattice over $\Int$ of signature $(n,2)$ with $n\geq 1$. We will write $Q$ for the quadratic form on $L$, and $[\_,\_]_Q$ for the associated bilinear form. Then one can associate with $L$ a Shimura variety $\Sh(L)$. It is a smooth Deligne-Mumford stack over $\Rat$ such that, as complex orbifolds, we have:
\[
\Sh(L)(\Comp)=G_L(\Rat)\backslash\bigl(X_L\times G_L(\Adele_f)/K_L\bigr).
\]
Here, $X_L$ is the space of oriented negative definite planes in $L_{\Real}$, $G_L$ is the reductive $\Rat$-group $\SO(L_{\Rat})$, and $K_L\subset G_L(\Adele_f)$ be the discriminant kernel of $L_{\widehat{\Int}}$: the largest sub-group of $\SO(L)(\widehat{\Int})$ that acts trivially on the discriminant $\on{disc}(L)=\dual{L}/L$, where $\dual{L}\subset L_{\Rat}$ is the dual lattice.

For the rest of this section, we will make the simplifying hypothesis that $L$ contains a hyperbolic plane: That is, we will assume that there exist isotropic elements $e,f\in L$ such that $[e,f]_Q=1$. Then we will have (use strong approximation for the Spin cover of $G_L$):
\[
 \Sh(L)(\Comp)=\Gamma_L\backslash X_L,
\]
where $\Gamma_L\subset\SO(L)(\Int)$ is the discriminant kernel.

Every compact open sub-group $K\subset K(L)$ determines a finite \'etale cover $\Sh_K(L)\to\Sh(L)$, defined over $\Rat$, with
\[
 \Sh_K(L)(\Comp)=G_L(\Rat)\backslash\bigl(X_L\times G_L(\Adele_f)/K\bigr).
\]
If $K$ is \defnword{neat}\footnote{This means that, for every $g\in G_L(\Adele_f)$, the discrete group $G_L(\Rat)\cap gKg^{-1}$ is torsion-free.}, then $\Sh_K$ is a smooth quasi-projective variety over $\Rat$.

\subsection{}\label{ortho:subsec:sheaves}
We will now show that, in the language of Section~\ref{sec:hodge}, $\Sh(L)$ carries a canonical family of $\Int$-motives $\bm{L}$. For details, cf.~\cite{mp:reg}*{\S~3}.

To begin, the Betti realization $\bm{L}_B$ will just be the local system on $\Sh(L)(\Comp)$ attached to the tautological representation $L$ of $\Gamma^+_L$. Note that $\bm{L}_B$ is equipped with a canonical symmetric bilinear form, which equips it with an injective map $\bm{L}_B\to\dual{\bm{L}}_B$ into its dual local system. The finite local system $\dual{\bm{L}}_B/\bm{L}_B$ with its $\underline{\Rat/\Int}$-valued quadratic form is canonically isomorphic to the constant sheaf $\underline{\dual{L}/L}$ over $\Sh(L)(\Comp)$. Furthermore, the determinant $\det(\bm{L}_B)$ is also identified with the constant sheaf $\underline{\det(L)}$ with its natural quadratic form.

For any prime $\ell$, the $\ell$-adic local system attached to $\bm{L}_B$ has a canonical descent over $\Sh(L)$, which we denote by $\bm{L}_{\ell}$. In fact, if $K(\ell^n)\subset K_L$ is the largest sub-group acting trivially on $L/\ell^nL$, then $\bm{L}_{\ell}$ is pro-represented over $\Sh(L)$ by the inverse system of finite \'etale covers $\{\Sh_{K(\ell^n)}(L)\}_{n\geq 1}$

The analytic vector bundle with integrable connection $\bm{L}_B\otimes\Reg{\Sh(L)_{\Comp}}^{\an}$ has a canonical algebraization $\bm{L}_{\dR,\Comp}$ over $\Sh(L)_{\Comp}$, which also descends canonically to a vector bundle with integrable connection $\bm{L}_{\dR,\Rat}$ over $\Sh(L)$. This vector bundle has the additional structure of a three-step filtration by vector sub-bundles:
\[
0=F^2\bm{L}_{\dR,\Rat}\subset F^1\bm{L}_{\dR,\Rat}\subset F^0\bm{L}_{\dR,\Rat}=(F^1\bm{L}_{\dR,\Rat})^{\perp}\subset F^{-1}\bm{L}_{\dR,\Rat}=\bm{L}_{\dR,\Rat}.
\]
Here, $F^1\bm{L}_{\dR,\Rat}\subset\bm{L}_{\dR,\Rat}$ is isotropic of rank $1$.

In fact, the pair $(\bm{L}_B,F^\bullet\bm{L}_{\dR,\Comp})$ forms a polarized variation of $\Int$-Hodge structures of weight $0$ over $\Sh(L)(\Comp)$. At each point of $\Sh(L)(\Comp)$, it gives rise to a $\Int$-Hodge structure with Hodge numbers $h^{-1,1}=h^{1,-1}=1,h^{0,0}=n$.

This allows us to give a moduli-theoretic description of $\Sh(L)$: Suppose that $T$ is a smooth complex analytic space and $f:T\to\Sh(L)(\Comp)$ is a map of smooth complex analytic stacks. We can attach to it the polarized variation of $\Int$-Hodge structures $(f^*\bm{L}_B,F^\bullet f^*\bm{L}^{\an}_{\dR,\Comp})$, and canonical identifications of the sheaves $f^*\dual{\bm{L}}_B/f^*\bm{L}_B$ and $\det(f^*\bm{L}_B)$ with the constant sheaves $\underline{\dual{L}/L}$ and $\underline{\det(L)}$, respectively. This gives us (cf.~\cite{milne:motive}*{3.10}):
\begin{prp}\label{ortho:prp:moduli}
The above process gives us a canonical equivalence between the category of maps of analytic stacks $T\to\Sh(L)(\Comp)$ and the category of tuples
\[
(\bm{U},F^\bullet\bigl(\bm{U}\otimes_{\Int}\Reg{T}\bigr),\eta,\beta),
\]
where:
\begin{itemize}
  \item $\biggl(\bm{U},F^\bullet\bigl(\bm{U}\otimes_{\Int}\Reg{T}\bigr)\biggr)$ is a polarized variation of $\Int$-Hodge structures over $T$ with constant Hodge numbers $h^{-1,1}=h^{1,-1}=1,h^{0,0}=n$.
  \item $\eta:\underline{\dual{L}/L}\xrightarrow{\simeq}\dual{\bm{U}}/\bm{U}$ and $\beta:\underline{\det(L)}\xrightarrow{\simeq}\det(\bm{U})$ are isomorphisms of sheaves of abelian groups over $T$, compatible with the natural pairings.
\end{itemize}
The local system $\bm{U}$ with its polarization pairing must satisfy the following additional condition: For every point $t\in T$, there exists an isometry of quadratic lattices $L\xrightarrow{\simeq}\bm{U}_t$ inducing $\eta_t$ and $\beta_t$.

If $T$ is a smooth algebraic variety over $\Comp$, this category can also be identified with the category of maps of algebraic stacks $T\to\Sh(L)_{\Comp}$.
\end{prp}
\qed

\subsection{}\label{ortho:subsec:motives}
The sheaves constructed above can all be viewed as realizations of a family of $\Int$-motives over $\Sh(L)$. This is essentially shown in \cite{mp:reg}*{\S~3}, with the additional inputs being (\ref{hodge:thm:abshodge}) and (\ref{hodge:thm:blasius}). The main idea is that there exists a finite \'etale cover $\widetilde{\Sh}(L)\to\Sh(L)$\footnote{This map is actually a bijection on geometric points.} attached to the central extension $\GSpin(L_{\Rat})\to G_L$, and an abelian scheme $A^{\KS}$ over $\widetilde{\Sh}(L)$ with an action of the Clifford algebra $C(L)$, called the \defnword{Kuga-Satake abelian scheme}. This abelian scheme arises from the natural action of $\GSpin(L_{\Rat})$ on $C(L_{\Rat})$ via left multiplication.

It has the following properties: For every point $s\to\Sh(L)$ with lift $s^{\on{sp}}\to\widetilde{\Sh}(L)$, the motive $h^1(A^{\KS}_{s^{\on{sp}}})\otimes \dual{\bigl(h^1(A^{\KS}_{s^{\on{sp}}})\bigr)}$, with its natural $\Int$-structure, depends only on $s$; we will denote it by $\bm{H}^{\otimes(1,1)}_s$. Furthermore, there is a natural idempotent operator $\pr_s$ on $\bm{H}^{\otimes(1,1)}_s$ such that, if $\bm{L}_s=\im\pr_s\subset\bm{H}^{\otimes(1,1)}_s$, then the various realizations of $\bm{L}_s$ are canonically identified with the fibers at $s$ of the sheaves $\bm{L}_B,\bm{L}_{\ell},\bm{L}_{\dR,\Rat}$ seen above.

In particular, we can view $\on{AH}(\bm{L}_s)$ as a space of ($C(L)$-equivariant) endomorphisms of $A^{\KS}_{s^{\on{sp}}}$, which we will refer to as \defnword{special endomorphisms}. If $s$ is a geometric point valued in a field embedded in $\Comp$, we will have:
\[
 \on{AH}(\bm{L}_s)=\bm{L}_{B,s}\cap F^0\bm{L}_{\dR,s,\Comp}\subset\bm{H}^{\otimes(1,1)}_{B,s}\cap F^0\bm{H}^{\otimes(1,1)}_{\dR,s,\Comp}=\End(A^{\KS}_{s^{\on{sp}}}).
\]
In general, given $T\to\Sh(L)$, we can define a `special endomorphism' over $T$ as follows: The endomorphism scheme $\underline{\End}(A^{\KS})$ over $\widetilde{\Sh}(L)$ has a canonical descent over $\Sh(L)$~\cite{mp:reg}*{5.24}. Write $\bm{E}$ for this descent; then the space of special endomorphisms $L(T)$ will consist of sections of $\bm{E}$ over $T$ whose fibers at every geometric point $s\to T$ lie in $\on{AH}(\bm{L}_s)$. Denote the space of special endomorphisms over $T$ by $L(T)$.

\subsection{}\label{ortho:subsec:integral}
From now on, we will further assume that the following condition: For every prime $p>2$, at least one of the following conditions holds:
\begin{itemize}
\item $L_{\Int_{(p)}}$ is maximal. That is, there is no bigger $\Int_{(p)}$-lattice in $L_{\Rat}$ on which the quadratic form is $\Int_{(p)}$-valued.
\item The $p$-primary part of the discriminant $\dual{L}/L$ is cyclic.
\end{itemize}

The main result of \cite{mp:reg} is:
\begin{thm}[\cite{mp:reg}*{8.1}]\label{ortho:thm:main}
Under these assumptions, $\Sh(L)$ admits an integral canonical model $\Ss(L)$ over $\Int[2^{-1}]$.
\end{thm}

The terminology here requires a bit of explanation. First, for every prime $p>2$, let $\Sh_p(L)$ be the pro-variety
\[
 \varprojlim_{K^p\subset G_L(\Adele_f^p)}\Sh_{K_{L,p}K^p}(L).
\]
It has a natural Hecke action by $G_L(\Adele_f^p)$. An \defnword{integral canonical model} for $\Sh_p(L)$ over $\Int_{(p)}$ is a locally healthy, regular pro-$\Int_{(p)}$-scheme $\Ss_p(L)$ with generic fiber $\Sh_p(L)$ satisfying the following extension property: For any locally healthy regular scheme $S$ over $\Int_{(p)}$, any map of generic fibers $S\otimes\Rat\to\Sh_p(L)$ extends (uniquely) to a map $S\to\Ss_p(L)$. A \defnword{smooth integral canonical model} $\Ss^{\on{sm}}_p(L)$ is defined similarly, except that we require it to be regular and formally smooth, and in the definition of the extension property, we restrict ourselves to schemes $S$ that are regular, formally smooth; cf.~\cite{mp:reg}*{8.5}. 

The integral canonical model (resp. smooth integral canonical model) is uniquely determined by these conditions. In particular, the action of $G_L(\Adele_f^p)$ on $\Sh_p(L)$ extends to an action on $\Ss_p(L)$ or $\Ss_p^{\on{sm}}(L)$, and so, for any compact open $K^p\subset G_L(\Adele_f^p)$, one obtains a model $\Ss_K(L)_{(p)}=\Ss_p(L)/K^p$ (resp. $\Ss^{\on{sm}}_K(L)_{(p)}=\Ss_p^{\on{sm}}(L)/K^p$) for $\Sh_{K_{L,p}K^p}(L)$. Here, we view the quotient in the category of algebraic stacks over $\Int_{(p)}$. 

We can now explain the meaning of (\ref{ortho:thm:main}). The model $\Ss(L)$ over $\Int[2^{-1}]$ is the unique one that satisfies the following property: For every $p>2$ such that $L_{\Int_{(p)}}$ is maximal (resp. is non-maximal with cyclic discriminant), $\Sh_p(L)$ admits an integral canonical model $\Ss_p(L)$ (resp. a smooth integral canonical model $\Ss^{\on{sm}}_p(L)$) over $\Int_{(p)}$ such that $\Ss(L)\otimes\Int_{(p)}=\Ss_{K_L}(L)_{(p)}$ (resp. $\Ss(L)\otimes\Int_{(p)}=\Ss^{\on{sm}}_{K_L}(L)_{(p)}$).

\subsection{}\label{ortho:subsec:integralmotives}
The integral model above carries a natural extension of the family of motives $\bm{L}$; cf. the discussion in \cite{mp:reg}*{8.7,8.10}. For simplicity, write $\Ss$ for the stack $\Ss(L)$. Quite formally, for any prime $\ell$, we can view $\bm{L}_{\ell}$ as an $\ell$-adic lisse sheaf over $\Ss[\ell^{-1}]$. Moreover, the de Rham realization $\bm{L}_{\dR,\Rat}$ extends to a vector bundle with integrable connection $\bm{L}_{\dR}$ over $\Ss$. The tautological isotropic line $F^1\bm{L}_{\dR,\Rat}$ also extends to an isotropic line $F^1\bm{L}_{\dR}\subset\bm{L}_{\dR}$.

Also, for any prime $p>2$, there is a natural $F$-crystal of vector bundles $\bm{L}_{\cris}$ over the crystalline site $(\Ss_{\Field_p}/\Int_p)_{\cris}$, whose Zariski realization over $\Ss_{\Int_p}$ is canonically identified with $\bm{L}_{\dR,\Int_p}$ as a vector bundle with integrable connection. The deformation theory of $\Ss$ is governed by the line $F^1\bm{L}_{\dR}$: Lifting a map $T\to\Ss$ over a first-order nilpotent thickening $T\into T'$ is equivalent to lifting the isotropic line $F^1\bm{L}_{\dR,T}$ over $T'$ (the lift of $\bm{L}_{\dR,T}$ over $T'$ being determined by its crystalline nature).

Suppose now that $E$ is a field of characteristic $0$ equipped with a discrete valuation $\nu$ with residue field $k$ of characteristic $p>2$, and suppose that we have an $E$-valued point $s:\Spec E\to\Sh(L)$ that extends to an $\Reg{E,(v)}$-valued point of $\Ss$. Then, in the notation of (\ref{hodge:prp:abelian}), the motive $\bm{V}_s\coloneqq\bm{L}_s\otimes\Rat$ belongs to the category $\mb{Mot}^{\circ}_{\on{Ab},\nu}(E)$. In fact, the crystalline realization of $\bm{V}_s$ can be identified with $\bm{L}_{\cris,s_0}\otimes\Rat$, where $s_0:\Spec k\to\Ss$ is the specialization of $s$, and $\bm{L}_{\cris,s_0}$ is the evaluation of $\bm{L}_{\cris}$ on the pro-nilpotent divided power thickening $\Spec k\into\Spec W(k)$. In particular, there exists a canonical crystalline comparison isomorphism
\begin{align}\label{ortho:eqn:criscomp}
\bm{L}_{p,\overline{s}}\otimes_{\Int_p}\Bcris\xrightarrow{\simeq}\bm{L}_{\cris,s_0}\otimes_{W(k)}\Bcris.
\end{align}
Here, $\overline{s}$ is a geometric point of $\Sh(L)$ lying above $s$.

The notion of a special endomorphism can also be extended to work over $\Ss$. The finite cover $\widetilde{\Sh}(L)\to\Sh(L)$ attached to the $\GSpin$ cover of $G_L$ extends to a finite \'etale map $\widetilde{\Ss}\to\Ss$ of (smooth) integral canonical models, and the Kuga-Satake abelian scheme also extends to an abelian scheme $A^{\KS}$ over $\widetilde{\Ss}$. If $\bm{H}_{\dR}$ is the relative first de Rham cohomology sheaf of $A^{\KS}$ over $\widetilde{\Ss}$, the inclusion $\bm{L}_{\dR,\Rat}\into\bm{H}_{\dR,\Rat}^{\otimes(1,1)}$ of vector bundles with flat connection over $\widetilde{\Sh}(L)$ extends to an inclusion $\bm{L}_{\dR}\subset\bm{H}_{\dR}^{\otimes(1,1)}$ over $\widetilde{\Ss}$. In turn, over $\widetilde{\Ss}_{\Field_p}$, this provides us an inclusion of $F$-crystals $\bm{L}_{\cris}\subset\bm{H}_{\cris}^{\otimes(1,1)}$.

If $s\to\widetilde{\Ss}$ is a point valued in a perfect field of characteristic $p>2$, we say that an endomorphism of $A^{\KS}_s$ is \defnword{special} if its crystalline realization in $\bm{H}^{\otimes(1,1)}_s$ actually lies in $\bm{L}_{\cris,s}$. This automatically implies that its $\ell$-adic realizations, for $\ell\neq p$, lie in $\bm{L}_{\ell,s}$; cf.~\cite{mp:reg}*{5.12}. In general, given an $\widetilde{\Ss}$-scheme $T$, we say that an endomorphism of $A^{\KS}_T$ is special if its restriction over every geometric point of $T$ is special.

Just as in the characteristic $0$ situation, the endomorphism sheaf of $A^{\KS}$ descends to a sheaf $\bm{E}$ over $\Ss$ along with the inclusion of crystals $\bm{L}_{\cris}\subset\bm{H}^{\otimes(1,1)}_s$. This allows us to speak of the group of `special endomorphisms' $L(T)$ over any $\Ss$-scheme $T$ even if the abelian scheme $A^{\KS}$ does not descend over $T$; cf.~\cite{mp:reg}*{8.13 }.

\subsection{}\label{ortho:subsec:embedding}
Suppose that we have a maximal quadratic lattice $L'$ of signature $(n+r,2)$ and an isometric embedding $L\into L'$ mapping $L$ onto a direct summand of $L'$. This gives rise to maps $\Ss\to\Ss(L')$ and $\widetilde{\Ss}\to\widetilde{\Ss}(L')$. There is now an additional notion of a special endomorphism over any $\Ss$-scheme $T$ arising from its induced structure as an $\Ss(L')$-scheme. Denote by $L'(T)$ the space consisting of this latter kind of special endomorphism.

Let $\Lambda=L^{\perp}\subset L'$. Then the relationship between $L(T)$ and $L'(T)$ can be described as follows~\cite{mp:reg}*{8.12}:
\begin{prp}\label{ortho:prp:specialemb}
There is a canonical isometric embedding $\Lambda\subset L'(\Ss)$, such that, for any $\Ss$-scheme $T$, we have a natural isometry:
\[
 L(T)\xrightarrow{\simeq}\Lambda^{\perp}\subset L'(T).
\]
This is compatible with isometries of sheaves:
\begin{align*}
  \bm{L}_{\ell}&\xrightarrow{\simeq}\Lambda^{\perp}\subset\bm{L}'_{\ell}\rvert_{\Ss[\ell^{-1}]},\text{ for any prime $\ell$;}\\
  \bm{L}_{\dR}&\xrightarrow{\simeq}\Lambda^{\perp}\subset\bm{L}'_{\dR}\rvert_{\Ss};\\
  \bm{L}_{\cris}&\xrightarrow{\simeq}\Lambda^{\perp}\subset\bm{L}'_{\dR}\rvert_{(\Ss_{\Field_p}/\Int_p)_{\cris}},\text{ for any prime $p>2$.}
\end{align*}
\end{prp}
\qed

\section{The Kuga-Satake period map over $\Int\left[\frac{1}{2}\right]$}\label{sec:ks}

In this section, we will study the classical period map for the moduli of K3 surfaces and show that it has an extension over $\Int[2^{-1}]$ with good properties. This allows us to extend the Kuga-Satake construction for K3 surfaces over fields of characteristic $p\neq 2$.

\subsection{}
Let $L_d$ be the quadratic lattice from (\ref{k3:subsec:level}): This is maximal at all primes $p>2$ such that $p^2\nmid d$ and has cyclic discriminant. So the theory of Section~\ref{sec:ortho} gives us an integral canonical model $\Ss(L_d)$ for $\Sh(L_d)$ over $\Int[2^{-1}]$. 

Over $\mathsf{M}_{2d,\Comp}^{\an}$, we have a natural isometric trivialization 
\[
\eta:\underline{\on{disc}(L_d)}\xrightarrow{\simeq}\on{disc}(\bm{P}^2_B).
\]
Indeed, for any point $s\to\mathsf{M}_{2d,\Comp}$, there is a canonical isometry:
\[
 \eta_s:\on{disc}(L_d)\xrightarrow{\simeq}\frac{N}{L_d\oplus\langle e-df\rangle}\xrightarrow{\simeq}\frac{\bm{H}^2_{B,s}}{\bm{P}^2_{B,s}\oplus\langle\on{ch}(\lambda)\rangle}\xrightarrow{\simeq}\on{disc}(\bm{P}^2_{B,s}),
\]
induced by any isometry $N\xrightarrow{\simeq}\bm{H}^2_{B,s}$ carrying $e-df$ to $\on{ch}(\lambda)$. Now, $\eta$ is the unique global isometry that interpolates the $\eta_s$.

Let $\tilde{\mathsf{M}}_{2d}\to\mathsf{M}_{2d}$ be the two-fold finite \'etale cover parameterizing isometric trivializations $\underline{\det(L_d)\otimes\Int_2}\xrightarrow{\simeq}\det(\bm{P}^2_2)$ of the determinant of the primitive $2$-adic cohomology of the universal quasi-polarized K3 surface. We can identify $\tilde{\mathsf{M}}_{2d,\Comp}$ with the space of isometric trivializations $\underline{\det(L_d)}\xrightarrow{\simeq}\det(\bm{P}^2_B)$ of the determinant of the primitive Betti cohomology.

Applying (\ref{ortho:prp:moduli}), we obtain:
\begin{prp}\label{ks:prp:torelli}
There is a natural period map
\[
 \iota_{\KS,\Comp}:\tilde{\mathsf{M}}_{2d,\Comp}\to\Sh(L_d)_{\Comp}
\]
attached to the tuple $(\bm{P}^2_B,F^\bullet\bm{P}^2_{\dR,\Comp},\eta,\beta)$, where $\beta$ is the tautological trivialization of $\det(\bm{P}^2_B)$ over $\tilde{\mathsf{M}}_{2d,\Comp}$.
\end{prp}
\qed
  
Cf. also~\cite{rizov:cm}*{Prop. 2.5} for a similar construction for (a finite cover of) $\mathsf{M}^{\circ}_{2d,\Comp}$, and~\cite{maulik}*{5.7} for its extension over the quasi-polarized locus.

\begin{prp}\label{ks:prp:hodgemotives}
For every point $s\in\tilde{\mathsf{M}}_{2d}(\Comp)$, there is a canonical isomorphism of $\Int$-motives:
\[
 \bm{L}_{\iota_{\KS,\Comp}(s)}(-1)\xrightarrow{\simeq}\bm{P}^2_s.
\]
\end{prp}
\begin{proof}
  This is shown as in the proof of \cite{dmos}*{II.6.26(d)}. Here are some more details: To begin, from the very construction of $\iota_{\KS,\Comp}$ there exists a canonical isometry
\[
 \alpha_B:\iota_{\KS,\Comp}^*\bm{L}_B(-1)\xrightarrow{\simeq}\bm{P}^2_{B}
\]
of polarized variations of $\Int$-Hodge structures over $\tilde{\mathsf{M}}^{\an}_{2d,\Comp}$. We can view this as a section of the variation of $\Int$-Hodge structures $\bigl(\iota_{\KS,\Comp}^*\bm{H}_B^{\otimes(1,1)}\otimes\bm{P}^2_B\bigr)(1)$. After replacing $\tilde{\mathsf{M}}_{2d}$ by a finite \'etale cover $T$, we can view $\bm{H}_B^{\otimes(1,1)}$ as the relative cohomology sheaf of a family of abelian varieties.

As in \emph{loc. cit.}, we can show by hand that $\alpha_{B,s}$ is absolutely Hodge when $\bm{\mathcal{X}}_s$ is a Kummer K3. Now we can appeal to Principle B of \cite{dmos}*{Ch. I}, which states that a horizontal Hodge cycle (on a family of smooth projective varieties over a smooth connected variety) that is absolutely Hodge at one point is absolutely Hodge everywhere. To apply this, we have to show that every connected component of $T$ contains a Kummer point. Since $\mathsf{M}_{2d,\Comp}$ is irreducible (cf.~\ref{ks:cor:moduli}), it suffices to exhibit a single Kummer surface over $\Comp$ equipped with a primitive quasi-polarization of degree $2d$.

Let $A$ be an abelian surface over $\Comp$ equipped with a polarization $\lambda$ of degree $2d$. Then the Kummer surface $X$ attached to $A$ is constructed as follows: One takes the blow-up $\widetilde{A}$ of the $2$-torsion in $A$, and then quotients $\widetilde{A}$ by the action of the canonical lift $\iota$ of the involution $[-1]$ on $A$ given by multiplication by $-1$. Any polarization on $A$ gives rise to an ample class $\lambda\in\on{NS}(A)$ and the pull-back of $2\lambda=\lambda+[-1]^*\lambda$ over $\widetilde{A}$ descends to a quasi-polarization $\xi\in\on{NS}(X)$. Moreover, if the polarization is of degree $d^2$, then by Riemann-Roch~\cite{mumford:abvar}*{III.16}, $\lambda$ has self-intersection $2d$, and, since $\widetilde{A}\to X$ is a degree $2$ map of smooth surfaces, $\xi$ has self-intersection $2d$ as well.

So, to finish, we have to construct an abelian surface $A$ with a primitive polarization of degree $d^2$. For this, take $A=E\times E$, with $E$ an elliptic curve, and the polarization to be the endomorphism $f\times (f\circ[d])$, where $f:E\xrightarrow{\simeq}\dual{E}$ is the canonical polarization of $E$.
\end{proof}

\begin{corollary}[Rizov]\label{ks:cor:kugasatakechar0}
$\iota_{\KS,\Comp}$ descends to a map
\[
 \iota_{\KS,\Rat}:\tilde{\mathsf{M}}_{2d,\Rat}\to\Sh(L_d).
\]
\end{corollary}
\begin{proof}
This is essentially \cite{rizov:cm}*{3.16} (cf. also~\cite{maulik}*{5.7}). Rizov shows that the map descends over $\Rat$ by proving the existence of a dense set of `CM points', for which the reciprocity law is compatible with Shimura-Taniyama reciprocity for CM points on the canonical model $\Sh(L_d)$.

But we will provide a different proof, using the theory of motives for absolute Hodge cycles. It is enough to see that, for every $\sigma\in\Aut(\Comp)$, $\iota_{\KS,\Comp}\circ\sigma=\sigma\circ\iota_{\KS,\Comp}$. For this, from (\ref{ortho:prp:moduli}), it is enough to see that both maps induce the same tuples (up to isomorphism) over $\tilde{\mathsf{M}}_{2d,\Comp}$. This is easy to deduce from the following consequence of (\ref{ks:prp:hodgemotives}): For every $s\in\tilde{\mathsf{M}}_{2d,\Comp}$, there are canonical isomorphisms of $\Int$-Hodge structures:
\[
 \bm{L}_{\sigma(\iota_{\KS,\Comp}(s))}(-1)\xrightarrow{\simeq}\bm{P}^2_{\sigma(s)}\xrightarrow{\simeq}\bm{L}_{\iota_{\KS,\Comp}(\sigma(s))}(-1).
\]
\end{proof}

\subsection{}\label{ks:subsec:motivic}
For the sake of convenience, given any sheaf $F$ over $\Sh(L_d)$ (with respect to any of the natural Grothendieck topologies), we will denote its pull-back along $\iota_{\KS,\Comp}$ again by the same letter $F$. This will apply in particular to the various realizations of the family of $\Int$-motives $\bm{L}$.

Via the de Rham comparison isomorphism, $\alpha_B$ gives rise to a canonical isometry of polarized filtered vector bundles with flat connection:
\[
 \alpha_{\dR,\Comp}:\bm{L}_{\dR,\Comp}(-1)\xrightarrow{\simeq}\bm{P}^2_{\dR,\Comp}.
\]
That this isometry is algebraic follows from \cite{deligne:eqsdiff} and the fact that both flat bundles have regular singularities along the boundary divisor in a suitable compactification of $\tilde{\mathsf{M}}^{\circ}_{2d,\Comp}$.

Via Artin's comparison isomorphisms, for any prime $\ell$, we also obtain compatible isometries of polarized local systems on $\tilde{\mathsf{M}}_{2d,\Comp}$:
\begin{align*}
 \alpha_{\ell}:\bm{L}_{\ell}(-1)&\xrightarrow{\simeq}\bm{P}^2_{\ell}\\
\end{align*}

\begin{prp}\label{ks:prp:descent}
\mbox{}
\begin{enumerate}
\item\label{descent:etale}For each prime $\ell$, the isometry $\alpha_{\ell}$ is defined over $\tilde{\mathsf{M}}_{2d,\Rat}$ (and hence over $\tilde{\mathsf{M}}_{2d,\Int[(2\ell)^{-1}]}$).
\item\label{descent:derham}The isomorphism $\alpha_{\dR,\Comp}$ descends to an isometry
\[
 \alpha_{\dR,\Rat}:\bm{L}_{\dR,\Rat}(-1)\xrightarrow{\simeq}\bm{P}^2_{\dR,\Rat}
\]
of filtered polarized vector bundles with flat connection over $\tilde{\mathsf{M}}_{2d,\Rat}$.
\item\label{descent:motive}For every point $s:\Spec F\to\tilde{\mathsf{M}}_{2d,\Rat}$, there is a canonical isometry of $\Int$-motives $\bm{L}_s(-1)\xrightarrow{\simeq}\bm{P}^2_s$. In particular, $\bm{P}^2_s$ is a motive in $\mb{Mot}_{\on{Ab}}(F)$ with $\Int$-structure.
\item\label{descent:derhamcomp}If $\nu:F\to\Int$ is a discrete valuation on $F$, then the isomorphism $\bm{L}_s(-1)\xrightarrow{\simeq}\bm{P}^2_s$ is a map of motives in $\mb{Mot}_{\on{AD},v}(F)$ with $\Int$-structure (cf.~\ref{hodge:subsec:admotives}).
\end{enumerate}
\end{prp}
\begin{proof}
To prove (\ref{descent:etale}), it is sufficient to show that the $\ell$-adic sheaf $\bm{P}^2_{\ell}$ does not admit any non-trivial isometries over $\tilde{\mathsf{M}}_{2d,\Comp}$ with trivial determinant. This follows from the fact that the attached monodromy representation is irreducible---a fact that can be deduced from the openness of the period map $\iota^{\KS}_{\Comp}$; cf.~\cite{deligne:k3weil}*{6.4}.

By (\ref{ks:prp:hodgemotives}), given a point $s\in\tilde{\mathsf{M}}_{2d}(\Comp)$, the isometry of Hodge structures
\[
 \alpha_{B,s}:\bm{L}_{B,s}(-1)\xrightarrow{\simeq}\bm{P}^2_{B,s}
\]
is absolutely Hodge. If we are now given a point $s\in\tilde{\mathsf{M}}_{2d}(F)$, where $F$ is a field of characteristic $0$ that is embeddable in $\Comp$, using (\ref{descent:etale}) for any $\ell$ and (\ref{hodge:prp:basechange}), we find that there exists a unique isometry of polarized $\Int$-motives
\[
  \alpha_s:\bm{L}_{s}(-1)\xrightarrow{\simeq}\bm{P}^2_{s}
\]
such that, for any embedding $\tau:F\into\Comp$, it induces the realizations $\alpha_{B,\tau(s)}$, $\alpha_{\ell,\tau(s)}$ and $\alpha_{\dR,\tau(s)}$. This shows (\ref{descent:motive}).

Applying (\ref{descent:motive}) to the generic points of $\tilde{\mathsf{M}}_{2d,\Rat}$, we get (\ref{descent:derham}).

(\ref{descent:derhamcomp}) now follows from the argument used for the proof of \cite{blasius}*{3.1(3)}; cf. also the proof of (\ref{ks:prp:hodgemotives}).
\end{proof}

\begin{prp}\label{ks:prp:kugasatakecharp}
$\iota^{\KS}_{\Rat}$ extends to a map
\[
  \iota^{\KS}:\tilde{\mathsf{M}}_{2d}\rightarrow \Ss(L_d)
\]
\end{prp}
\begin{proof}
  For every compact open $K\subset K_{L_d}$, write $\tilde{\mathsf{M}}_{2d,K,\Rat}$ for the pull-back of $\mathsf{M}_{2d,K,\Rat}$ over $\tilde{\mathsf{M}}_{2d,\Rat}$. Then the map $\iota^{\KS}_{\Rat}$ lifts naturally to a map $\iota^{\KS}_{K,\Rat}:\tilde{\mathsf{M}}_{2d,K,\Rat}\to\Sh_K(L_d)$.

  Fix a prime $p>2$, and write $\tilde{\mathsf{M}}_{2d,K_{L_d,p}}$ for the $\Int_{(p)}$-scheme defined as the inverse limit
  \[
    \varprojlim_{K^p\subset G_d(\Adele_f^p)}\tilde{\mathsf{M}}_{2d,K_{L_d,p}K^p,\Int_{(p)}}.
  \]  
  Then $\tilde{\mathsf{M}}_{2d,K_{L_d,p},\Rat}$ admits a map $\iota^{\KS}_{K_{L_d,p},\Rat}$ to $\Sh_p(L_d)$ giving rise to $\iota^{\KS}_{K,\Rat}$ at each finite level. Therefore, if $\nu_p(d)\leq 1$ (resp. $\nu_p(d)>1$), by the extension property of the integral canonical model $\Ss_p(L_d)$ (resp. the smooth integral canonical model $\Ss^{\on{sm}}_p(L_d)$), $\iota^{\KS}_{K_{L_d,p},\Rat}$ extends uniquely to a map
  \[
   \iota^{\KS}_{K_{L_d,p}}:\tilde{\mathsf{M}}_{2d,K_{L_d,p}}\to\Ss_p(L_d)\text{ (or $\Ss^{\on{sm}}_p(L_d)$)}.
  \]
  In turn, this gives us an extension over $\Int_{(p)}$:
  \[
  \iota^{\KS}_{\Int_{(p)}}:\tilde{\mathsf{M}}_{2d,\Int_{(p)}}\to\Ss(L_d)_{\Int_{(p)}}.
  \]

\end{proof}
\comment{
\begin{corollary}\label{ks:cor:connected}
Every geometrically connected component of $\mathsf{M}_{2d,K,\Field_p}$ is the specialization of a geometrically connected component of $\mathsf{M}_{2d,K,\Rat}$. In particular, $\mathsf{M}_{2d,\Field_p}$ is a geometrically irreducible Deligne-Mumford stack.
\end{corollary}
\begin{proof}
We first observe that $\mathsf{M}^{\ord,\circ}_{2d,K,\Int_{(p)}}$ is dense fiber-by-fiber in $\mathsf{M}_{2d,K,\Int_{(p)}}$, whose special fiber is normal. So it suffices to prove the result for the former. But now $\mathsf{M}^{\ord,\circ}_{2d,K,\Int_{(p)}}$ is an open sub-scheme of $\Ss^{\on{pr}}_{d,K}$, whose special fiber is again normal. So the result follows from (\ref{ortho:cor:conncomp}).

For the second, it is now enough to show that $\mathsf{M}_{2d,\Comp}$ is an irreducible stack, which is well-known; cf. for example~\cite{barth_peters_vdv}*{Ch. VIII}.
\end{proof}
}
The main result of this section is:
\begin{thm}\label{ks:thm:kugasatakecharp}
The map $\iota^{\KS}$ is \'etale.
\end{thm}

We will need a few preliminaries before we can prove (\ref{ks:thm:kugasatakecharp}), the main input being (\ref{ks:prp:cris}) below. The proof will appear right below that of \emph{loc. cit.}

\begin{lem}\label{ks:lem:strongdivext}
Let $k$ be a perfect field of characteristic $p>2$, and let $W=W(k)$. For $s:\Spec W\to\tilde{\mathsf{M}}_{2d}$, the map
\[
 \alpha_{\dR,s_{\Rat}}:\bm{L}_{\dR,s_{\Rat}}(-1)\xrightarrow{\simeq}\bm{P}^2_{\dR,s_{\Rat}}
\]
is an isomorphism of $F$-isocrystals.
\end{lem}
\begin{proof}
This is shown in~\cite{ogus:duke}*{\S~7}, but we can provide a different proof with the technology of Section~\ref{sec:hodge}.

  Let $\overline{s}_{\Rat}$ be a geometric point above $s_{\Rat}$ valued in an algebraic closure $\overline{W}_{\Rat}$. Then we have comparison isomorphisms:
  \begin{align*}
    \bm{L}_{p,\overline{s}_{\Rat}}\otimes\Bcris&\xrightarrow{\simeq}\bm{L}_{\dR,s}\otimes\Bcris;\\
    \bm{P}^2_{p,\overline{s}_{\Rat}}\otimes\Bcris&\xrightarrow{\simeq}\bm{P}^2_{\dR,s}\otimes\Bcris.
  \end{align*}

  We also have a natural isomorphism of $\Gal(\overline{W}_{\Rat}/W_{\Rat})$-representations:
  \begin{align*}
   \alpha_{p,\overline{s}_{\Rat}}:\bm{L}_{p,\overline{s}_{\Rat}}(-1)&\xrightarrow{\simeq}\bm{P}^2_{p,\overline{s}_{\Rat}}
  \end{align*}
  arising from an isomorphism of motives $\bm{L}_{s_{\Rat}}(-1)\xrightarrow{\simeq}\bm{P}^2_{s_{\Rat}}$. It now follows from (\ref{ks:prp:descent})(\ref{descent:derhamcomp}) that $\alpha_{\dR,s_{\Rat}}$ is exactly the map obtained from $\alpha_{p,\overline{s}_{\Rat}}$ via the crystalline comparison isomorphisms. In particular, it is $F$-equivariant.
\end{proof}

\begin{lem}\label{ks:lem:ordinary}
Suppose that ${s}:\Spec W\to\tilde{\mathsf{M}}_{2d}$ is a lift of an ordinary point $s_0:\Spec k\to\tilde{\mathsf{M}}_{2d}$. Then $\alpha_{\dR,{s}_{\Rat}}$ carries $\bm{L}_{\dR,{s}}(-1)$ onto $\bm{P}^2_{\dR,{s}}$.
\end{lem}
\begin{proof}
\comment{By the argument in \cite{maulik}*{6.15}, it is enough to show that, for every map $R\to W$ attached to a lift ${s}:\Spec W\to\mathsf{M}_{2d}$ of an ordinary point of $\mathsf{M}$, the induced map $\alpha_{\dR,W_{\Rat}}$ carries $\bm{L}_{\dR,W}(-1)$ onto $\bm{P}^2_{\dR,W}$.}

First, by the Dieudonn\'e-Manin classification~\cite{manin} (cf. also~\cite{katz:dwork}*{2.1}), $\bm{L}_{\dR,{s}}(-1)$ (resp. $\bm{P}^2_{\dR,{s}}$) admits a canonical largest $F$-stable direct summand $\bm{L}_{\dR,{s},0}(-1)$ (resp. $\bm{P}^2_{\dR,{s},0}$) to which $F$ restricts to an isomorphism (this is the \defnword{slope $0$} part). In fact, this sub-$F$-crystal must be of rank $1$. It suffices to check this for $\bm{P}^2_{\dR,{s}}(-1)$, for which cf.~\cite{ogus:height}*{p. 327}.

Let $U_{\bullet}\bm{L}_{\dR,{s}}(-1)$ be the three-step ascending filtration on $\bm{L}_{\dR,{s}}(-1)$ determined by
\[
U_0\bm{L}_{\dR,{s}}(-1)=\bm{L}_{\dR,{s},0}(-1)\;;\;U_1\bm{L}_{\dR,{s}}=\bm{L}^{\perp}_{\dR,{s},0}(-1).
\]
Analogously define an ascending filtration $U_{\bullet}\bm{P}^2_{\dR,{s}}$ on $\bm{P}^2_{\dR,{s}}$. These are the canonical slope filtrations and are in particular preserved by the $F$-equivariant map $\alpha_{\dR,{s}_{\Rat}}$ after changing scalars to $W_{\Rat}$.

Let $\Comp_p$ be the completion of $\overline{s}_{\Rat}$ and let $\Reg{\Comp_p}$ be its ring of integers. It can now be deduced from \cite{bloch_kato}*{9.6} that there are ascending $\Gamma$-stable filtrations $U_\bullet\bm{L}_{p,\overline{s}_{\Rat}}(-1)$ and $U_\bullet\bm{P}^2_{p,\overline{s}_{\Rat}}$ that satisfy the following conditions:
\begin{enumerate}
  \item The crystalline comparison isomorphisms (for both $\bm{L}$ and $\bm{P}^2$) respect the $U$-filtrations on either side (in fact, one can \emph{define} the $U$-filtrations on the \'etale side to be the unique ones that satisfy this property).
  \item The canonical $\Gamma$-equivariant isomorphism
      \[
      \eta_{p,\overline{s}_{\Rat}}:\bm{L}_{p,\overline{s}_{\Rat}}(-1)\xrightarrow{\simeq}\bm{P}^2_{p,\overline{s}_{\Rat}}
      \]
      respects $U$-filtrations.
  \item For each $n\in\Int$, we have $\Gamma$-equivariant isomorphisms compatible with the comparison isomorphisms:
\begin{align}\label{ks:eqn:grisom}
    \gr^U_n\bm{L}_{p,\overline{s}_{\Rat}}(-1)\otimes \Reg{\Comp_p}&\xrightarrow{\simeq}\gr^U_n\bm{L}_{\dR,{s}}(-1)\otimes\Reg{\Comp_p}(-n);\\
    \gr^U_n\bm{P}^2_{p,\overline{s}_{\Rat}}\otimes \Reg{\Comp_p}&\xrightarrow{\simeq}\gr^U_n\bm{P}^2_{\dR,{s}}\otimes\Reg{\Comp_p}(-n).
  \end{align}
\end{enumerate}

Now, $\gr^U\alpha_{\dR,{s}_{\Rat}}$ has to be compatible with $\gr^U\eta_{p,\overline{s}_{\Rat}}$ under the isomorphisms in (\ref{ks:eqn:grisom}), and $\eta_{p,\overline{s}_{\Rat}}$ carries $\bm{L}_{p,\overline{s}_{\Rat}}(-1)$ onto $\bm{P}^2_{p,\overline{s}_{\Rat}}$. Therefore, since $\Reg{\Comp_p}$ is faithfully flat over $W$, we find that $\gr^U\alpha_{\dR,{s}_{\Rat}}$ must carry $\gr^U\bm{L}_{\dR,{s}}(-1)$ onto $\gr^U\bm{P}^2_{\dR,{s}}$.

By the strong divisibility of $\bm{L}_{\dR,{s}}(-1)$ (cf.~\cite{mp:reg}*{4.8}), we must have
\[
U_0\bm{L}_{\dR,{s}}(-1)\cap F^1\bm{L}_{\dR,{s}}(-1)=0.
\]
Since $U_0\bm{L}_{\dR,{s}}(-1)$ is an isotropic line (this can be seen, for example, from the fact that $F(f)\circ F(f)=p^2(f\circ f)$, for any $f\in\bm{L}_{\dR,{s}}(-1)$), we obtain a splitting of $U_\bullet\bm{L}_{\dR,{s}}(-1)$:
\[
 \bm{L}_{\dR,{s}}(-1)=F^2\bm{L}_{\dR,{s}}(-1)\oplus(F^2\bm{L}_{\dR,{s}}(-1)\oplus U_0\bm{L}_{\dR,{s}}(-1))^{\perp}\oplus U_0\bm{L}_{\dR,{s}}(-1).
\]
We similarly define a splitting for $U_\bullet\bm{P}^2_{\dR,{s}}$, and the construction shows that these splittings are compatible with $\alpha_{\dR,{s}_{\Rat}}$. Therefore, $\alpha_{\dR,{s}_{\Rat}}$ must indeed carry $\bm{L}_{\dR,{s}}(-1)$ onto $\bm{P}^2_{\dR,{s}}$.
\end{proof}

\comment{
We immediately obtain:
\begin{corollary}\label{ks:cor:intprop}
If $s$ is an ordinary point, then there is an open neighborhood $U$ of $s$ in $\mathsf{M}_{2d,\Int_p}$ such that $\alpha_{\dR,U_{\Rat_p}}$ carries $\bm{L}_{\dR,U}(-1)$ onto $\bm{P}^2_{\dR,U}$.
\end{corollary}
\qed}
The following result, which exhibits the \emph{integral} crystalline nature of the Kuga-Satake construction is the chief ingredient in the proof of (\ref{ks:thm:kugasatakecharp}); cf~\cite{maulik}*{6.8} for an essentially equivalent statement, but with stronger hypotheses on $d$ and $p$.
\begin{prp}\label{ks:prp:cris}
The isometry
\[
 \alpha_{\dR,\Rat}:\bm{L}_{\dR,\Rat}(-1)\xrightarrow{\simeq}\bm{P}^2_{\dR}\rvert_{\mathsf{M}_{2d,\Rat}}
\]
extends to an isometry (necessarily unique)
\[
 \alpha_{\dR}:\bm{L}_{\dR}(-1)\xrightarrow{\simeq}\bm{P}^2_{\dR}
\]
of vector bundles over $\tilde{\mathsf{M}}_{2d}$ with integrable connection. It carries $F^1\bm{L}_{\dR}(-1)$ onto $F^2\bm{P}^2_{\dR}$.
\end{prp}
\begin{proof}
It is enough to extend $\alpha_{\dR,\Rat}$ as a map of vector bundles, since the other requirements can be checked over $\Rat$. 

Fix a prime $p>2$ and an affine open $U=\Spec R\subset\tilde{\mathsf{M}}_{2d,\Int_p}$ such that the restrictions of $\bm{L}_{\dR}$ and $\bm{P}^2_{\dR}$ to $U$ are both trivial. We now follow the argument from \cite{maulik}*{6.15}. First, represent $\alpha_{\dR,U_{\Rat_p}}$ by a matrix with values in $R_{\Rat}$. We claim that the entries of this matrix lie in $R$. Indeed, let $a\in R_{\Rat}$ be a matrix entry, and let $m\in\Int_{\geq 0}$ be minimal such that $a'=p^ma\in R$. By (\ref{ks:lem:ordinary}), for any $W(\overline{\Field}_p)$-valued point of $U$ with ordinary reduction, the image of $a$ in $W(\overline{\Field}_p)_{\Rat}$ lies in $W$. In particular, if $m>1$, then the value of $a'$ at any such point would lie in $pW(\overline{\Field}_p)$. Since ordinary points are dense in $U_{\Field_p}$, this implies that $a'$ must lie in $pR$, which is a contradiction. Indeed, otherwise, the image of $a'$ in $R_{\Field_p}$ would be a non-zero global function on $U_{\Field_p}$ that vanishes at a dense set of points.

An analogous argument shows that the matrix entries of $\alpha_{\dR,U_{\Rat_p}}^{-1}$ also lie in $R$, thus proving that $\alpha_{\dR,U_{\Rat}}$ extends to an isometry $\alpha_{\dR,U}:\bm{L}_{\dR,U}(-1)\xrightarrow{\simeq}\bm{P}^2_{\dR,U}$.

From this, the proposition follows.
\end{proof}

\subsection{}
For any smooth point $s\in\tilde{\mathsf{M}}_{2d,\Field_p}(\overline{\Field}_p)$, let $R$ be the completion of the local ring at $s$. Set $W=W(\overline{\Field}_p)$, and choose a lift $j:R\to W$. Equip $R$ with an endomorphism $\varphi$ lifting the $p$-power Frobenius on $R_{\Field_p}$ such that $\sigma\circ j=j\circ\varphi$. The restrictions of $\bm{L}_{\dR}(-1)$ and $\bm{P}^2_{\dR}$ to $\Spec R$ give rise to $F$-crystals over $R$ in the terminology of \cite{katz:dwork}*{1.3}. We will denote these $F$-crystals by $\bm{L}_{\dR,R}(-1)$ and $\bm{P}^2_{\dR,R}$, respectively.  The reductions of $\bm{L}_{\dR,R}(-1)$ to $\bm{P}^2_{\dR,R}$ along $j$ will be denoted $\bm{L}_{\dR,W}(-1)$ and $\bm{P}^2_{\dR,W}$, respectively.

\begin{lem}\label{ks:cor:strongdivext}
$\alpha_{\dR,\Spec R}$ is an isomorphism of filtered $F$-crystals over $R$. \comment{It carries $\bm{L}_{\dR,R}(-1)$ onto $\bm{P}^2_{\dR,R}$ if and only if $\alpha_{\dR,W_{\Rat}}$ carries $\bm{L}_{\dR,W}(-1)$ onto $\bm{P}^2_{\dR,W}$. }
\end{lem}
\begin{proof}
\comment{Once we show that $\alpha_{\dR,R_{\Rat}}$ is a map of $F$-isocrystals, the second assertion will follow from (\ref{ks:prp:fisocrys}).}
  Let $R^{\an}_{\Rat}$ be the ring of functions on the rigid analytic space over $W_{\Rat}$ attached to $\Spf R$. Then, by \cite{katz:dwork}*{3.1}, there exist unique $F$-equivariant, horizontal isomorphisms that reduce to the identity along $j^*$:
  \begin{align*}
    \bm{L}_{\dR,W}(-1)\otimes_WR^{\an}_{\Rat}&\xrightarrow{\simeq}\bm{L}_{\dR,R^{\an}_{\Rat}}(-1);\\
    \bm{P}^2_{\dR,W}\otimes_WR^{\an}_{\Rat}&\xrightarrow{\simeq}\bm{P}^2_{\dR,R^{\an}_{\Rat}}.
  \end{align*}
  Here, we equip the left hand sides with the constant connection $1\otimes\on{d}$ and the constant $F$-structures induced from the ones on $\bm{L}_{\dR,W}(-1)$ and $\bm{P}^2_{\dR,W}$.

  Since $\alpha_{\dR,R_{\Rat}}$ is horizontal for the connection, it now suffices to check that the induced isomorphism
  \[
  \alpha_{\dR,W_{\Rat}}:\bm{L}_{\dR,W_{\Rat}}(-1)\xrightarrow{\simeq}\bm{P}^2_{\dR,W_{\Rat}}
  \]
  is a map of $F$-isocrystals over $W_{\Rat}$, and this follows from (\ref{ks:lem:strongdivext}).
\end{proof}

Let $T\to\tilde{\mathsf{M}}_{2d,\Field_p}$ be an \'etale map with $T$ a scheme. Then one can also consider the crystalline realization $\bm{P}^2_{\cris,T}$ of the primitive cohomology of the universal family $\bm{\mathcal{X}}_T\to T$: This will be a crystal of vector bundles over $(T/\Int_p)_{\cris}$. At the same time, one also has the crystal $\bm{L}_{\cris,T}(-1)$ over $(T/\Int_p)_{\cris}$. In fact, both these crystals have the additional structure of an $F$-crystal. That is, if $\on{Fr}_T:T\to T$ is the absolute Frobenius on $T$, then we have natural maps $\on{Fr}_T^*\bm{P}^2_{\cris}\to\bm{P}^2_{\cris}$ and $\on{Fr}_T^*\bm{L}_{\cris}(-1)\to\bm{L}_{\cris}(-1)$.
\begin{corollary}\label{ks:cor:criscomp}
$\alpha_{\dR}$ induces a canonical isomorphism of $F$-crystals
\[
 \bm{L}_{\cris,T}(-1)\xrightarrow{\simeq}\bm{P}^2_{\cris,T}.
\]
\end{corollary}
\begin{proof}
If $T$ is smooth (this is always the case unless $\nu_p(d)=1$), then this follows from (\ref{ks:prp:cris}) and (\ref{ks:cor:strongdivext}). Indeed, working locally if necessary, we can assume that $T$ lifts to a smooth map $\widetilde{T}\to\tilde{\mathsf{M}}_{2d,\Int_p}$. Now, one can use the classical equivalence of categories between crystals on $T$ and vector bundles over $\widetilde{T}$ with integrable connections.

Suppose now that $\nu_p(d)=1$ and that $T$ is not smooth. Then, according to (\ref{k3:cor:regular}), $T$ has at worst isolated singular points with quadratic singularities. Let $T^{\on{sm}}\subset T$ be the smooth locus. The result now follows from the fact that restriction of crystals of vector bundles from $(T/\Int_p)_{\cris}$ to $(T^{\on{sm}}/\Int_p)_{\cris}$ is a fully faithful operation.
\end{proof}

\begin{proof}[Proof of (\ref{ks:thm:kugasatakecharp})]
It is enough to show that, for every prime $p>2$, and every closed point $s\in\tilde{\mathsf{M}}_{2d}(\overline{\Field}_p)$, the induced map of complete local $\Int_p$-algebras
\[
 \widehat{\Rg}_{\Ss(L_d),\iota^{\KS}(s)}\rightarrow\widehat{\Rg}_{\tilde{\mathsf{M}}_{2d},s}
\]
is an isomorphism. For simplicity denote this map by $R\to R'$.

Both $R$ and $R'$ are complete local Noetherian domains of the same dimension, namely $19$, so it is enough to show that the induced map of tangent spaces $t_R\rightarrow t_{R'}$ is an isomorphism. But, by (\ref{ks:prp:cris}), both $t_R$ and $t_{R'}$ can be canonically identified with the space
\[
\biggl\{\text{Isotropic lines $L\subset \bm{P}^2_{\dR,s}\otimes\overline{\Field}_p[\epsilon]$ lifting $F^2\bm{P}^2_{\dR,s}$}\biggr\}.
\]
Under these identifications, the map on tangent spaces is simply the identity. This can be checked, for example, by lifting to characteristic $0$.
\end{proof}

Fix a prime $p>2$. Given any neat compact open $K\subset K_{L_d}$ with $K_p=K_{L_d,p}$, $\mathsf{M}_{2d,K,\Int_{(p)}}$ admits a (non-canonical) section to $\tilde{\mathsf{M}}_{2d,\Int_{(p)}}$.
\begin{corollary}\label{ks:cor:open}
The induced map $\iota^{\KS}_K:\mathsf{M}^{\circ}_{2d,K,\Int_{(p)}}\to\Ss_K(L_d)_{(p)}$ is an open immersion.
\end{corollary}
\begin{proof}
  Clearly, $\iota^{\KS}_K$ is \'etale. Therefore, by \cite{laumon_m-b}*{16.5}, there exists a finite $\Ss_K(L_d)_{(p)}$-scheme $\mathscr{Z}$ and a factoring as below, where the top arrow is a dense open immersion.
  \begin{diagram}
    \mathsf{M}^{\circ}_{2d,K,\Int_{(p)}}&\rInto&\mathscr{Z}\\
    &\rdTo_{\iota^{\KS}}&\dTo\\
    &&\Ss_K(L_d)_{(p)}.
  \end{diagram}
  Since $\Ss_K(L_d)_{(p)}$ is a normal scheme, it is sufficient to show that $\iota^{\KS}_{K,\Rat}$ restricted to $\mathsf{M}^{\circ}_{2d,K,\Rat}$ is an open immersion. Indeed, it would then follow that $\mathscr{Z}\to\Ss_K(L_d)_{(p)}$ is in fact an isomorphism onto a union of connected components of $\Ss_K(L_d)_{(p)}$. So the restriction of $\iota^{\KS}_K$ to $\mathsf{M}^{\circ}_{2d,K,\Int_{(p)}}$ must be an open immersion.
  
  To finish, it suffices to show that $\iota^{\KS}_{K,\Comp}$ is an open immersion. This is essentially the global Torelli theorem for K3 surfaces, for which there are many proofs in the literature; cf.~\cites{ps_shafarevich,burns_rapoport,looijenga_peters,friedman}. For a good summary and yet another proof, cf.~\cite{huybrechts}. Our ad\'elic formulation can be found in \cite{rizov:cm}*{Prop. 2.10}.
\end{proof}

\begin{corollary}\label{ks:cor:moduli}
For every $p>2$, $\mathsf{M}^{\circ}_{2d,\Field_p}$ is a quasi-projective Deligne-Mumford stack over $\Field_p$. Moreover, the Hodge bundle $\omega=F^2\bm{H}^2_{\dR,\Field_p}$ is ample over $\mathsf{M}^{\circ}_{2d,\Field_p}$. If $\nu_p(d)\leq 1$, then $\mathsf{M}^{\circ}_{2d,\Field_p}$ is geometrically irreducible.
\end{corollary}
\begin{proof}
  The quasi-projectivity is immediate from (\ref{ks:thm:kugasatakecharp}) and the quasi-projectivity of $\Ss(L_d)_{\Field_p}$.

  To show ampleness of $\omega$ it suffices by (\ref{ks:prp:cris}) to show the ampleness of $F^1\bm{L}_{\dR,\Field_p}$. This follows from \cite{mp:reg}*{4.18}.

  Suppose now that $\nu_p(d)\leq 1$. Then $L_d$ is maximal at $p$, and the result follows from \cite{mp:reg}*{8.3}, (\ref{ks:cor:open}), and the fact that $\on{M}^{\circ}_{2d,\Comp}$ is irreducible (as can be seen from global Torelli and the complex analytic uniformization).
\end{proof}

\begin{thm}\label{ks:thm:ksconst}
Given any field $k$ of odd characteristic $p$ and a polarized K3 surface $(X,\xi)$ over $k$ of degree $2d$, there exists a finite separable extension $k'/k$ and an abelian variety $A$ over $k'$, the \defnword{Kuga-Satake} abelian variety, equipped with an action of the Clifford algebra $C(L_d)$, which enjoys the following additional properties:
\begin{enumerate}
  \item\label{ksconst:h1}Fix a separable closure $k^{\on{sep}}$ of $k'$. For every prime $\ell\neq p$, there exists an isomorphism of $\Int_{\ell}$-modules
  \[
   H^1_{\et}\bigl(A_{k^{\on{sep}}},\Int_{\ell}\bigr)\xrightarrow{\simeq}C\bigl(PH^2_{\et}(X_{k^{\on{sep}}},\Int_{\ell}(1))\bigr).
  \]
  Here, the right hand side denotes the Clifford algebra attached to the quadratic lattice $PH^2_{\et}\bigl(X_{k^{\on{sep}}},\Int_{\ell}(1)\bigr)$. Moreover, let $k^p$ be a perfect closure of $k'$; then there exists an isomorphism of $W(k^p)$-modules
  \[
   H^1_{\cris}\bigl(A_{k^p}/W(k^p)\bigr)\xrightarrow{\simeq}C\bigl(PH^2_{\cris}(X_{k^p}/W(k^p))(1)\bigr).
  \]
  \item\label{ksconst:end}For all primes $\ell\neq p$, the algebra
      \[
      C(L_d)\otimes\Int_{\ell}\subset\End\bigl(H^1_{\et}\bigl(A_{k^{\on{sep}}},\Int_{\ell}\bigr)\bigr)
      \]
      is Galois-equivariantly identified with $C\bigl(PH^2_{\et}(X_{k^{\on{sep}}},\Int_{\ell}(1))\bigr)$ acting on itself by right translation via the isomorphism in (\ref{ksconst:h1}). Similarly, $C(L_d)\otimes W(k^p)$ is $F$-equivariantly identified with $C\bigl(PH^2_{\cris}(X_{k^p}/W(k^p))(1)\bigr)$
  \item\label{ksconst:ph2}The action of $C\bigl(PH^2_{\et}(X_{k^{\on{sep}}},\Int_{\ell}(1))\bigr)$ on itself by left translations induces, via (\ref{ksconst:h1}), a Galois-equivariant embedding
      \[
       PH^2_{\et}\bigl(X_{k^{\on{sep}}},\Int_{\ell}(1)\bigr)\subset\End_{C(L_d)}\bigl(H^1_{\et}\bigl(A_{k^{\on{sep}}},\Int_{\ell}\bigr)\bigr).
      \]
      Similarly, there is an $F$-equivariant embedding
      \[
      PH^2_{\cris}\bigl(X_{k^p}/W(k^p)\bigr)(1)\subset\End_{C(L_d)}\bigl(H^1_{\cris}(A_{k^p}/W(k^p))\bigr).
      \]
  \item\label{ksconst:special}Let $L(A)\subset\End(A)$ be the sub-space of endomorphisms whose cohomological realizations lie in the image of $PH^2_{\et}\bigl(X_{k^{\on{sep}}},\Int_{\ell}(1)\bigr)$ for all $\ell\neq p$, as well as in the image of $PH^2_{\cris}\bigl(X_{k^p}/W(k^p)\bigr)(1)$. Then there is a natural identification
      \[
       \Pic(X_{k'})\supset\langle\xi\rangle^{\perp}\xrightarrow{\simeq}L(A)
      \]
      compatible with all cohomological realizations.
\end{enumerate}
\end{thm}
\begin{proof}
After replacing $k$ by a finite separable extension if necessary we can assume that $(X,\xi)$ arises from a point $s\in\tilde{\mathsf{M}}_{2d}(k)$ that lifts to a $k$-valued point of $\widetilde{\Ss}(L_d)$. To this lift, we can attach the Kuga-Satake abelian variety $A^{\KS}_s$ with properties (\ref{ksconst:h1}), (\ref{ksconst:end}) and (\ref{ksconst:ph2}). The integral crystalline compatibility here follows from (\ref{ks:cor:criscomp}).

It still remains to show (\ref{ksconst:special}). For this we observe that, given a special endomorphism $f\in L(A^{\KS}_s)=L(s)$, the deformation space of the triple $(X,\xi,f)$ admits a flat component. Indeed, by (\ref{ks:thm:kugasatakecharp}), we can identify the complete local ring of $\tilde{\mathsf{M}}_{2d}$ at $s$ with that of $\Ss(L_d)$, and so the claim follows from \cite{mp:reg}*{8.15}.

We see therefore that there exists a lift $(\widetilde{X},\widetilde{\xi},\widetilde{f})$ over a characteristic $0$ field $F$ attached to a point $\widetilde{s}\in\tilde{\mathsf{M}}_{2d}(F)$ lifting $s$. Here we have:
\begin{align*}
 \Pic\bigl(\widetilde{X}\bigr)\supset\langle\widetilde{\xi}\rangle^{\perp}&\xrightarrow[\simeq]{\text{Lefschetz (1,1)}}\on{AH}\bigl(\bm{P}^2_{\overline{s}}\bigr)\\
 &=\on{AH}(\bm{L}_{\overline{s}})\cap\End(A^{\KS}_{\widetilde{s}})=L\bigl(A^{\KS}_{\widetilde{s}}\bigr).
\end{align*}
See (\ref{hodge:subsec:rstruc}) and (\ref{ortho:subsec:motives}) for the notation. In particular, there is a unique element of $\langle\widetilde{\xi}\rangle^{\perp}$ mapping to $\widetilde{f}$ under this isomorphism. Reducing back over $k$ shows that there is a unique element of $\langle\xi\rangle^{\perp}\subset\Pic(X)$ that has the same cohomological realizations as $f$.

Repeating this step for all $f\in L(A^{\KS}_s)$ shows that we have an inclusion $L(A^{\KS}_s)\into\langle \xi\rangle^{\perp}$ compatible with cohomological realizations.

Similarly, given a class $\eta\in\langle\xi\rangle^{\perp}\subset\Pic(X)$, the deformation space of $(X,\xi,\eta)$ again admits a flat component. Repeating the same argument as above gives us an inclusion going the other way, and so finishes the proof.
\end{proof}

\begin{rem}\label{ks:rem:ksconst}
In the literature (cf. for example~\cite{deligne:k3weil},~\cite{andre:shaf}), one usually finds the \emph{even} Clifford algebra in place of the full Clifford algebra that we have chosen to use. As in~\cite{charles}*{3.3}, we do this to ensure that the statement in (\ref{ksconst:ph2}) above is not too unwieldy.
\end{rem}

\begin{rem}\label{ks:rem:divtoend}
In fact, one can show more. For every map $T\to\mathsf{M}_{2d}$, we have a canonical identification:
\[
 L(T)=\langle\bm{\xi}\rangle^{\perp}_T\subset\Pic(\bm{\mathcal{X}}_T/T).
\]
Here, $L(T)$ is the space of special endomorphisms over $T$ viewed as a a scheme over $\Ss_d$

Indeed, the functors $T\mapsto L(T)$ and $T\mapsto\langle\bm{\xi}\rangle^{\perp}_T$ are both representable, unramified, and locally of finite type over $\mathsf{M}_{2d}$. Moreover, it is easy to deduce from the above argument that, for any field $k$, there is a canonical bijection between their $k$-valued points. In addition, one sees using deformation theory that, given a $k$-valued point of either stack, the complete local ring at that point is canonically isomorphic to the complete local ring at the associated $k$-valued point of the other stack. Using this and Artin approximation, one can glue together a canonical isomorphism from one stack to the other.
\end{rem}

\begin{rem}\label{ks:rem:deformation}
Notice that we did not need the full force of the \'etaleness of $\iota^{\KS}$ in the proof above. All we needed was for the intersection of the deformation space of a polarized K3 surface with the deformation space of a special endomorphism to admit a flat component. This weaker condition might still be checkable in situations where the Kuga-Satake period map is not expected to be \'etale, such as in the context of the Catanese-Ciliberto surfaces considered in \cite{lyons}.
\end{rem}

\subsection{}\label{ks:subsec:canonical}
In~\cite{rizov:kugasatake}*{4.2}, Rizov shows that, when $p\nmid d$, the Kuga-Satake construction is compatible with the theory of canonical lifts for ordinary varieties. This continues to hold in our more general situation. Suppose that $(X_0,\xi_0)$ is a polarized K3 surface over a perfect field $k$ of characteristic $p$, and suppose that $X_0$ is ordinary. Let $(X,\xi)$ be the canonical lift (cf. \emph{loc. cit.}) of $(X_0,\xi_0)$ over $W(k)$. After replacing $k$ by a finite extension, if necessary, we can assume that there is a Kuga-Satake abelian variety $A_0$ over $k$ attached to $(X_0,\xi_0)$, as in Theorem~\ref{intro:thm:ksconst}. There is also an algebraizable deformation $A$ of $A_0$ over $W(k)$ attached to the canonical lift $(X,\xi)$.
\begin{prp}\label{ks:prp:canonical}
$A_0$ is ordinary and $A$ is its canonical lift.
\end{prp}
\begin{proof}
  The proof of \cite{rizov:kugasatake}*{4.2.2} goes through verbatim. We recall it here briefly for the convenience of the reader. That $A_0$ is ordinary was already observed in the course of the proof of (\ref{ks:lem:ordinary}). Via Serre-Tate co-ordinates, the deformation space of $A_0$ is naturally identified with a formal torus $\mathfrak{T}$ over $W(k)$. Nygaard has shown in \cite{nygaard}*{2.7} that in this situation $A$ has to be isogenous to the canonical lift, implying that it corresponds to a torsion point of $\mathfrak{T}$. However, the only torsion point of $\mathfrak{T}$ defined over $W(k)$ is the identity, which corresponds to the canonical lift of $A_0$.
\end{proof}

\section{The Tate conjecture}\label{sec:special}

Let $L$ be a quadratic lattice as in Section~\ref{sec:ortho} satisfying the conditions of (\ref{ortho:subsec:integral}), and let $\Ss,\widetilde{\Ss}$ be the attached integral models over $\Int[2^{-1}]$ of $\Sh(L)$ and $\widetilde{\Sh}(L)$, respectively. Fix a prime $p>2$.

\subsection{}
Suppose that we have $s\in\widetilde{\Ss}(\overline{\Field}_p)$, and suppose that $s$ in fact arises from a point $s_0$ defined over the finite field $\Field_{p^r}$. Then, for each $\ell\neq p$ and each $m$ such that $r\vert m$, the $p^m$-power Frobenius $\Fr_m$ acts on $\bm{H}_{\ell,s}$ and on $\bm{L}_{\ell,s}$. We will write $\bm{V}_{\ell,s}$ for the $\Rat_{\ell}$-vector space $\bm{L}_{\ell,s}\otimes\Rat_{\ell}$.

For any $m\in\Int_{\geq 1}$, set $\Int_{p^r}=W(\Field_{p^m})$, and let $\Rat_{p^m}$ be its fraction field. We have the crystalline realization $\bm{L}_{\cris,s}$: this is a $\Int_{p^r}$-module of endomorphisms of the $\Int_{p^r}$-module $\bm{H}_{\cris,s}$. Set $\bm{V}_{\cris,s}=\bm{L}_{\cris,s}\left[\frac{1}{p}\right]$. For each $t$ such that $r\mid t$, let 
\[
\bigl(\Rat_{p^m}\otimes_{\Rat_{p^r}}\bm{V}_{\cris,s}\bigr)^{\bm{F}_s=1}\subset\Rat_{p^m}\otimes_{\Rat_{p^r}}\bm{V}_{\cris,s}
\]
denote the $\Rat_p$-subspace of $\bm{F}_s$-equivariant endomorphisms. For any prime $\ell$, set
\[
 r_{\ell}=\begin{cases}
   \dim_{\Rat_{\ell}}(\varinjlim_{r\vert m}\bm{V}_{\ell,s}^{\Fr_m=1})\text{, if $\ell\neq p$};\\
   \dim_{\Rat_p}(\varinjlim_{r\vert m}(\Rat_{p^m}\otimes_{\Rat_{p^r}}\bm{V}_{\cris,s})^{\bm{F}_s=1})\text{, if $\ell=p$}.
 \end{cases}
\]
From now on, we will maintain:
\begin{assump}[$\ell$-independence]\label{special:assump:ellind}
  $r_{\ell}$ is independent of $\ell$.
\end{assump}

\begin{rem}\label{special:rem:kisin}
  This assumption should always hold by results of Kisin~\cite{kis:langrap}. Also, by \cite{katz_messing}, it will hold if one can realize $\bigl\{\{\bm{V}_{\ell,s}\}_{\ell\neq p},\bm{V}_{\cris,s}\bigr\}$ as the family of cohomological realizations of a motive over $\Field_{p^r}$.
\end{rem}

\begin{thm}\label{special:thm:formtate}
\mbox{}
\begin{enumerate}[itemsep=0.12in]
\item\label{formtate:ell}If $\ell\neq p$, the natural map
  \begin{align*}
    L(A^{\KS}_s)\otimes\Rat_{\ell}&\rightarrow\varinjlim_{r\vert m}\bm{V}_{\ell,s}^{\Fr_m=1}
  \end{align*}
  is an isometry of $\Rat_{\ell}$-quadratic spaces.
\item\label{formtate:p}The natural map
\begin{align*}
  L(A^{\KS}_s)\otimes\Rat_p&\rightarrow\varinjlim_{r\vert m}(\Rat_{p^m}\otimes\Rat_{p^r}\bm{V}_{\cris,s})^{\bm{F}_s=1}
\end{align*}
is an isometry of $\Rat_p$-quadratic spaces.
\end{enumerate}
\end{thm}

\begin{rem}\label{special:rem:formtate}
Given our standing assumption (\ref{special:assump:ellind}), each of the assertions of the theorem is equivalent to the following statement: $\rk L(A^{\KS}_s)=r$, where $r=r_{\ell}$, for one (hence any) prime $\ell$.
\end{rem}

The proof of this theorem will be given below following (\ref{special:rem:kisinproof}). As noted in the introduction, the flexibility of working with arbitrary orthogonal Shimura varieties is important to our method. It enables us to make the following crucial reduction:
\begin{lem}\label{special:lem:red1}
We can assume that $L$ is self-dual at $p$ and that $L(A^{\KS}_s)\neq 0$.
\end{lem}
\begin{proof}
  Choose any non-zero positive definite quadratic space $\Lambda$ over $\Int$ such that $V'=V\oplus\Lambda_{\Rat}$ admits a lattice $L'$ that is self-dual at $p$. It is always possible to find such a $\Lambda$; cf.~\cite{mp:reg}*{6.1}. Attached to this is a map of (smooth) integral canonical models $\widetilde{\Ss}\to\widetilde{\Ss}(L')$. Let $\widetilde{s}$ be the image of $s$ in $\widetilde{\Ss}(L')$. Set
  \[
 \widetilde{r}_{\ell}=\begin{cases}
   \dim_{\Rat_{\ell}}(\varinjlim_{r\vert m}\widetilde{\bm{V}}_{\ell,\widetilde{s}}^{\Fr_m=1})\text{, if $\ell\neq p$};\\
   \dim_{\Rat_p}(\varinjlim_{r\vert m}(\Rat_{p^m}\otimes_{\Rat_{p^r}}\widetilde{\bm{V}}_{\cris,s})^{\bm{F}_s=1}\text{, if $\ell=p$}.
 \end{cases}
\]
Then, by (\ref{ortho:prp:specialemb}), we have, for all $\ell$, $\widetilde{r}_{\ell}=r_{\ell}+\rk\Lambda$. Therefore, the assumption (\ref{special:assump:ellind}) holds for $s\in\widetilde{\Ss}(\overline{\Field_p})$ if and only if it holds for $\widetilde{s}\in\widetilde{\Ss}(L')(\overline{\Field}_p)$.

Moreover, by \emph{loc. cit.}, we have
\[
   L(A^{\KS}_s)=\Lambda^{\perp}\subset L(\widetilde{A}^{\KS}_{\widetilde{s}}).
\]
So we find that (\ref{special:thm:formtate}) holds for $\widetilde{s}$ if and only if it holds for $s$.
\end{proof}

\subsection{}
Following this lemma, we can and will maintain the assumptions that $L$ is self-dual at $p$ and that $L(A^{\KS}_s)\neq 0$.

For $\ell\neq p$ and all $m\in\Int_{>0}$ such that $r\mid m$, let $I_{\ell,m}\subset G_{\ell}\coloneqq\GSpin(\bm{V}_{\ell,s})$ be the commutant of $\Fr_m$. Since $\Fr_m$ is a semi-simple element, $I_{\ell,m}$ is a reductive sub-group of $\GSpin(\bm{V}_{\ell,s})$. In fact, for $m$ large enough $I_{\ell,m}$ does not depend on $m$. From now on, we will fix $m$ such that $I_{\ell,m}=I_{\ell,m'}$, for all $m'\geq m$ with $r\vert m'$. Note that, for such $m$ and $m'$, we have
\[
 \bm{V}_{\ell,s}^{\Fr_{m'}=1}=\bm{V}_{\ell,s}^{\Fr_m=1}.
\]
We will write $I_{\ell}$ for the group $I_{\ell,m}$.

\begin{lem}\label{special:lem:irreducible}
For every $\ell\neq p$, $\bm{V}_{\ell,s}^{\Fr_m=1}$ is an absolutely irreducible representation of $I_{\ell}$.
\end{lem}
\begin{proof}
  Let $q=p^m$. Fix $\ell\neq p$, and let $1,\alpha_1^{\pm 1},\ldots,\alpha_r^{\pm 1}\in\overline{\Rat}_{\ell}$ be the distinct eigenvalues of $\Fr_m$ acting on $\bm{V}_{\ell,s}$. Since $\Fr_m$ is semi-simple, for $\ell\neq p$, the image of $I_{\ell}\otimes\overline{\Rat}_{\ell}$ in $\SO\bigl(\bm{V}_{\ell,s,\overline{\Rat}_{\ell}}\bigr)$ is the product
  \[
   \SO\bigl(\bm{V}^{\Fr_m=1}_{\ell,s,\overline{\Rat}_{\ell}}\bigr)\times\prod_{i=1}^r\GL\bigl(\bm{V}^{\Fr_m=\alpha_i}_{\ell,s,\overline{\Rat}_{\ell}}\bigr).
  \]
  From this description, the lemma is immediate.
\end{proof}

\subsection{}
Let $\underline{\Aut}^{\circ}(A^{\KS}_s)$ be the group scheme of units in the ring $\End(A^{\KS}_s)\otimes\Rat$: this is an algebraic group over $\Rat$. For $\ell\neq p$, there is a natural embedding of algebraic $\Rat_{\ell}$-groups
\[
  i_{\ell}:\underline{\Aut}^{\circ}(A^{\KS}_s)\otimes_{\Rat}\Rat_{\ell}\hookrightarrow \GL(\bm{H}_{\ell,s}\otimes\Rat_{\ell})
\]
defined by the functoriality of $\ell$-adic homology.

Similarly, if $\Rat_p^{\nr}$ is the fraction field of $W(\overline{\Field}_p)$, we have a natural embedding of algebraic $\Rat^{\nr}_p$-groups
\[
  i_p:\underline{\Aut}^{\circ}(A^{\KS}_s)\otimes_{\Rat}\Rat^{\nr}_p\hookrightarrow\GL(\bm{H}_{\cris,s,\Rat_p^{\nr}}).
\]

Let $I\subset\underline{\Aut}^{\circ}(A^{\KS}_s)$ be the largest closed sub-group that maps into $G_{\ell}$ under $i_{\ell}$ for each $\ell\neq p$, and into $\GSpin(\bm{V}_{\cris,s})$ under $i_p$.

We will need the following proposition:
\begin{prp}[Kisin]\label{special:prp:kisin}
Suppose that $\ell\neq p$ is a prime such that $G_{\ell}$ is split and such that all the roots of the characteristic polynomial of $\Fr_m$ are contained in $\Rat_{\ell}$.\footnote{Such a prime always exists; in fact, the set of such primes has positive density.} Then the natural map $I_{\Rat_{\ell}}\to I_{\ell}$ is an isomorphism.
\end{prp}
\begin{proof}\label{special:rem:kisinproof}
This is proven in~\cite{kis:langrap} via a group-theoretic reinterpretation of Tate's original argument for the main theorem of~\cite{tate:end}.

We only give a sketch here. $I$ is easily seen to be reductive, since it preserves a polarization on $A^{\KS}_s$. So it suffices to show that $I_{\Rat_{\ell}}$ contains a Borel sub-group of $I_{\ell}$. For this, using the splitness of $I_{\ell}$ (which holds by our hypothesis that $G_{\ell}$ is split), and a little further argument, it is enough to prove that the $\ell$-adic manifold $I(\Rat_{\ell})\backslash I_{\ell}(\Rat_{\ell})$ is compact. Choose a neat compact open $K\subset K_L$ with $K_p=K_{L,p}=\SO(L)(\Int_p)$. We can assume that $s$ is an $\bb{F}_{p^m}$-valued point of the finite \'etale cover $\widetilde{\Ss}_K(L)_{(p)}$ of $\widetilde{\Ss}(L)_{\Int_{(p)}}$.

Set $U_{\ell}=K_{\ell}\cap I_{\ell}(\Rat_{\ell})$; clearly, it suffices to show that the double coset space
\[
 I(\Rat_{\ell})\backslash I_{\ell}(\Rat_{\ell})/U_{\ell}
\]
is finite.

For this, we observe that the natural action of $I_{\ell}(\Rat_{\ell})$ on the set of $\Fr_m$-stable $\Int_{\ell}$-lattices within $\bm{H}_{\ell,s}\otimes\Rat_{\ell}$, along with the correspondence between these lattices and abelian varieties over $\Field_{p^m}$ isogenous to $A^{\KS}_s$, provides us with a map
\[
I(\Rat)\backslash I_{\ell}(\Rat_{\ell})/U_{\ell}\to\widetilde{\Ss}_K(L)_{(p)}(\Field_{p^m}).
\]
Using essential information from the construction of his integral canonical models, Kisin is able to show that this map is actually injective. This shows that the left hand side is finite and completes our proof.
\end{proof}

\begin{proof}[Proof of (\ref{special:thm:formtate})]
 For $\ell\neq p$, the map
 \[
  L(A^{\KS}_s)\otimes\Rat_{\ell}\rightarrow\bm{V}_{\ell,s}^{\Fr_m=1}
 \]
 is a map of $I\otimes\Rat_{\ell}$-representations, and so, by (\ref{special:prp:kisin}), for a particular choice of $\ell$, it is in fact a map of $I_{\ell}$-representations. But now, by (\ref{special:lem:irreducible}), $\bm{V}_{\ell,s}^{\Fr_m=1}$ is an irreducible representation of $I_{\ell}$. Since $L(A^{\KS}_s)\neq 0$, this implies that the map must be an isomorphism for this choice of $\ell$. By (\ref{special:rem:formtate}), this finishes the proof of the theorem.
\end{proof}

The following corollary is inspired from \cite{faltings:rp}*{\S 3}.
\begin{corollary}\label{special:cor:fingentate}
Suppose that the $\ell$-independence assumption (\ref{special:assump:ellind}) holds at every point in $\widetilde{\Ss}(\overline{\Field}_p)$. Let $s\to\widetilde{\Ss}$ be a point defined over a finitely generated extension of $\Field_p$, and let $\overline{s}\to\widetilde{\Ss}$ be a geometric point above $s$. Then, for each prime $\ell\neq p$, the natural map
\[
 L(A^{\KS}_s)\otimes\Rat_{\ell}\rightarrow\bm{V}_{\ell,\overline{s}}^{\Aut(k(\overline{s})/k(s))}
\]
is an isometry of $\Rat_{\ell}$-quadratic spaces.
\end{corollary}
\begin{proof}
We can assume that $k(s)=k(X)$ is the function field of a smooth, geometrically connected variety $X$ over $\Field_q$ equipped with an $\Field_q$-valued rational point $x_0$. We can also arrange things so that $s$ arises from a map $\widetilde{s}:X\to\widetilde{\Ss}$, and thus specializes to an $\Field_q$-valued point $s_0=\widetilde{s}\circ x_0$. By shrinking $X$ if necessary, we can further assume that
\[
 \End\bigl(A^{\KS}_{\widetilde{s}}\bigr)=\End\bigl(A^{\KS}_s\bigr).
\]
By the definition of specialness, we have:
\begin{align}\label{special:eqn:fingensp}
 L(A^{\KS}_s)=\End\bigl(A^{\KS}_{\widetilde{s}}\bigr)\cap L(A^{\KS}_{s_0})\subset\End\bigl(A^{\KS}_{s_0}\bigr).
\end{align}

Therefore from (\ref{special:thm:formtate}), we find, for a geometric point $\overline{s}_0$ lying above $s_0$:
\begin{align}\label{special:eqn:tensorql}
L(A^{\KS}_s)\otimes\Rat_{\ell}=\biggl(\End\bigl(A^{\KS}_{\widetilde{s}}\bigr)\otimes\Rat_{\ell}\biggr)\cap\bm{V}_{\ell,\overline{s}_0}\subset\End\bigl(\bm{H}_{\ell,\overline{s}_0}\bigr)\otimes\Rat_{\ell}.
\end{align}

By \cite{zarhin:tate}, we have, for any $\ell\neq p$:
\begin{align}\label{special:eqn:zarhin}
\End\bigl(A^{\KS}_s\bigr)\otimes\Rat_{\ell}\xrightarrow{\simeq}\End_{\Aut(k(\overline{s})/k(s))}\bigl(\bm{H}_{\ell,\overline{s}}\bigr)\otimes\Rat_{\ell}.
\end{align}

Combining this with (\ref{special:eqn:tensorql}) gives us the result.
\end{proof}
\comment{
Finally, we can prove an $\ell$-independence result for special endomorphisms.
\begin{corollary}\label{special:cor:ell-independence}
Suppose again that the $\ell$-independence assumption (\ref{special:assump:ellind}) holds at every point in $\Ss_K(\overline{\Field}_p)$. Let $T$ be an $\Ss_K$-scheme and suppose that $f\in\End(A^{\KS}_T)_{(p)}$. Then $f$ is special if and only if it is $\ell_0$-special for some prime $\ell_0$.
\end{corollary}
\begin{proof}
  Suppose that, for some prime $\ell_0$, $f$ is $\ell_0$-special. Then, by (\ref{special:lem:complex}) and (\ref{special:thm:formtate}), the locus where $f$ is special is an open and closed sub-scheme of $T$ that contains $T\otimes\Rat$ and $T(\overline{\Field}_p)$, and so must be all of $T$.
\end{proof}
}
\subsection{}\label{ks:subsec:mainproof}
We can now easily prove the Tate conjecture for K3 surfaces:
\begin{proof}[Proof of Theorem~\ref{intro:thm:main}]
After replacing $k$ by a finite separable extension, we can assume that $X$ admits a polarization $\xi$ of degree $2d$, and that $(X,\xi)$ corresponds to a point $s\in\mathsf{M}_{2d}(k)$. After a further finite separable extension of $k$, if necessary, we can assume that it lifts to a point in $\widetilde{\Ss}(L_d)(k)$.

If $k$ is finite, the theorem is immediate from (\ref{ks:thm:ksconst}) and (\ref{special:thm:formtate}). The required $\ell$-independence hypothesis (\ref{special:assump:ellind}) is valid in our case because of the obviously motivic origin of $\bm{V}_{\ell,s}=\bm{P}^2_{\ell,s}(1)$; cf.~(\ref{special:rem:kisin}).

For infinite $k$, the result follows easily from the proof of (\ref{special:cor:fingentate}), once we observe that the argument there only needs a smooth open neighborhood $U\subset\widetilde{\Ss}(L_d)$ of $s$ such that the $\ell$-independence hypothesis holds at some closed point of $U$.
\end{proof}

\subsection{}\label{special:subsec:cubicfourfolds}
We quickly sketch how the above ideas apply to cubic fourfolds. Let $M_0$ be the even rank $2$ $\Int$-lattice equipped with the bi-linear form represented by the matrix $\begin{pmatrix}
   2&1\\
   1&2
   \end{pmatrix}$. Let $M$ be the quadratic $\Int$-lattice:
\[
 M=E_8^{\oplus 2}\oplus U^{\oplus 2}\oplus M_0.
\]
This is a signature $(20,2)$ lattice that is maximal and is self-dual over $\Int[6^{-1}]$.

\comment{
Let $K_M\subset\SO(\widetilde{M})(\Adele_f)$ be the largest compact sub-group that stabilizes $\widetilde{M}_{\widehat{\Int}}$ and acts trivially on $m$ (this is the \defnword{discriminant kernel}); then $K_M$ is in fact a compact open sub-group of $\SO(M)(\Adele_f)$. For every prime $p\neq 2$, and every $K\subset K_M$ small enough with $K_p=K_{M,p}$, just as in Section~\ref{sec:ortho} (but even simpler, since we do not have to consider non-maximal lattices), the theory of \cite{mp:reg} now gives us an orthogonal Shimura variety $\Sh_K$ of the Shimura variety attached to $M$ and the level sub-group $K$ over $\Rat$, and an integral canonical healthy regular model $\Ss_K$ over $\Int_{(p)}$ (it is smooth over $\Int_{(p)}$ if $p\neq 3$).
}
Let $\mathsf{CF}$ be the moduli stack of cubic fourfolds over $\Int[2^{-1}]$.  Over $\Comp$, we have a Kuga-Satake map $\tilde{\mathsf{CF}}_{\Comp}\to\Sh(M)_{\Comp}$ constructed using primitive degree-$4$ cohomology, where, once again, $\tilde{\mathsf{CF}}$ is a two-fold cover of $\mathsf{CF}$ trivializing the determinant of primitive cohomology. Using the fact that this map is given via an absolutely Hodge correspondence~\cite{andre:shaf}*{\S~6}\footnote{This is proven in \emph{loc. cit.} via a monodromy argument, but one should also be able to prove it via Deligne's Principle B and working with amenable points in the moduli space, much as we did with Kummer points in (\ref{ks:prp:hodgemotives}).}, just as in (\ref{ks:cor:kugasatakechar0}), we can descend the Kuga-Satake map over $\Rat$: $\tilde{\mathsf{CF}}_{\Rat}\to\Sh(M)$.

Let $(+1)$ (resp. $(-1)$) be the self-dual odd positive (resp. negative) $\Int$-lattice of rank $1$, and set
\[
 M'=(+1)^{\oplus 21}\oplus(-1)^{\oplus 2}.
\]
This is a self-dual lattice of signature $(21,2)$. It is shown in \cite{hassett}*{2.1.2} that there exists $m\in M'$ with $m\cdot m=3$ such that $M$ is isometric to $\langle m\rangle^{\perp}\subset M'$. Then, for any $p>2$, just as we did for K3 surfaces in \ref{sec:k3}, we can define a notion of $K^p$-level structure for cubic fourfolds over $\Int_{(p)}$ using the lattice $M'$ and the distinguished element $m\in M'$. This gives us a finite \'etale cover $\tilde{\mathsf{CF}}_{K,\Int_{(p)}}\to \tilde{\mathsf{CF}}_{\Int_{(p)}}$. Since $\tilde{\mathsf{CF}}_{K,\Int_{(p)}}$ is smooth over $\Int_{(p)}$, we can again use the theory of integral canonical models to find a natural extension of the Kuga-Satake map over $\Int[2^{-1}]$: $\tilde{\mathsf{CF}}\to\Ss(M)$. We now have:

\begin{thm}\label{ks:thm:cubicfourfolds}
\mbox{}
\begin{enumerate}[itemsep=0.11in]
  \item\label{ks:cubicopen}The period map $\tilde{\mathsf{CF}}\to\Ss(M)$ is \'etale. For any $p>2$ and $K^p$ small enough, the map $\tilde{\mathsf{CF}}_{K,\Int_{(p)}}\to\Ss_K(M)_{(p)}$ is an open immersion.
  \item\label{ks:cubicksconst}Given any cubic fourfold $X$ over a field $k$ of odd characteristic, there exists a finite separable extension $k'/k$ and an abelian variety $A$ over $k'$ such that the numbered assertions of (\ref{ks:thm:ksconst}) hold with $PH^2$ replaced by $PH^4$ and $\Pic(X_{k'})$ replaced by $\on{CH}^2(X_{k'})$.
  \item\label{ks:cubictate}The Tate conjecture for cubic fourfolds holds in co-dimension $2$ over fields of odd characteristic. That is, given a cubic fourfold $X$ over a finitely generated field $k$ of odd characteristic with absolute Galois group $\Gamma=\Gal(k^{\on{sep}}/k)$, the $\ell$-adic cycle class map
      \[
       \on{CH}^2(X)\otimes\Rat_{\ell}\rightarrow H^4_{\et}\bigl(X_{k^{\on{sep}}},\Rat_{\ell}(2)\bigr)^{\Gamma}
      \]
      is an isomorphism for all $\ell\neq p$.
  \item\label{ks:cubicirred}$\mathsf{CF}_{\Field_p}$ is geometrically irreducible for every $p>2$.
\end{enumerate}
\end{thm}
\begin{proof}[Sketch of proof]
If we look back at the strategy used for K3 surfaces, we see that the main step is to show that the period map
\[
 \tilde{\mathsf{CF}}\rightarrow\Ss(M)
\]
is \'etale. Indeed, once we know this, the Torelli theorem for cubic fourfolds~\cite{voisin:cubic} will imply that the induced map $\tilde{\mathsf{CF}}_{K,\Int_{(p)}}\to\Ss_K(M)_{(p)}$ is an open immersion for $K^p$ small enough. The remaining statements are proven just as for K3 surfaces. We only note that, for the Tate conjecture, we have to appeal to the Hodge conjecture for co-dimension $2$ cycles on cubic fourfolds over $\Comp$, which is known; cf.~\cite{andre:shaf}*{Appendix 2} or~\cite{zucker:cubic}. This plays the same role for cubic fourfolds as Lefschetz (1,1) did for K3 surfaces.

To prove \'etaleness, we note that $\tilde{\mathsf{CF}}$ is smooth and that the tangent space at any point $s:\Spec k\to \tilde{\mathsf{CF}}$ attached to a cubic fourfold $X/k$ is given by:
\[
 \Def_{X}\bigl(k[\epsilon]\bigr)=\biggl\{\text{Isotropic lines $L\subset P^4_{\dR}(X/k)\otimes k[\epsilon]$ lifting $F^4H^4_{\dR}(X/k)$}\biggr\}.
\]
This is shown in \cite{levin}*{\S~3}. So, just as in the proof of (\ref{ks:thm:kugasatakecharp}), it is enough to prove the integral crystalline compatibility of the Kuga-Satake construction. We do this using the same strategy: prove it directly for ordinary cubic fourfolds as in (\ref{ks:lem:ordinary}) and then propagate it everywhere using the density of ordinary points as in (\ref{ks:prp:cris}).
\end{proof}
\comment{
\begin{rem}\label{ks:rem:charlesint}
In a recent pre-print~\cite{charles_pirutka}, Charles and Pirutka have shown the validity of the Tate conjecture for co-dimension $2$ cycles on cubic fourfolds (over finite fields of characteristic $p\geq 5$) with $\Int_{\ell}$-coefficients in place of $\Rat_{\ell}$-coefficients. This also follows from the above for $\ell\neq 2,3,p$ (allowing $p=3$), since we only have to check it for the corresponding space of special endomorphisms of the associated Kuga-Satake abelian variety.
\end{rem}
}

\appendix

\comment{
\section{Special endomorphisms}\label{app:special}

}

\end{document}